\documentclass[review,onefignum,onetabnum]{siamart190516}

\usepackage{amsmath,amsfonts,amssymb}
\usepackage{microtype}
\usepackage{graphicx}
\usepackage{subfigure}
\usepackage{booktabs} 
\usepackage{hyperref}
\usepackage{lscape}
\usepackage{enumitem}

\usepackage{url}

\usepackage{ulem,soul}

\usepackage{lipsum}
\usepackage{amsfonts}
\usepackage{graphicx}
\usepackage{epstopdf}
\usepackage{algorithmic}
\usepackage{float}
\usepackage{bm}
\ifpdf
  \DeclareGraphicsExtensions{.eps,.pdf,.png,.jpg}
\else
  \DeclareGraphicsExtensions{.eps}
\fi

\newsiamremark{remark}{Remark}
\newsiamremark{hypothesis}{Hypothesis}
\crefname{hypothesis}{Hypothesis}{Hypotheses}
\newsiamthm{claim}{Claim}

\newtheorem{Proposition}{Proposition}

\title{A Successive Two-stage Method for Sparse Generalized Eigenvalue Problems}

\author{Qia Li\thanks{School of Computer Science and Engineering, Guangdong Province Key Laboratory of Computational Science, Sun Yat-sen University, Guangzhou, Guangdong Province 510275, P.R.China (liqia@mail.sysu.edu.cn).} \and Jianmin Liao\thanks{Department of Mathematics, Syracuse University, Syracuse, NY 13244, USA (jliao21@syr.edu).} \and Lixin Shen\thanks{Department of Mathematics, Syracuse University, Syracuse, NY 13244, USA (lshen03@syr.edu).} \and Na Zhang\thanks{Corresponding author, Department of Applied Mathematics, College of Mathematics and Informatics, South China Agricultural University, Guangzhou, Guangdong Province 510642, P.R.China (nzhsysu@gmail.com).}
}

\usepackage{amsopn}

\ifpdf
\hypersetup{
	pdftitle={A Successive Two-stage Method for Sparse Generalized Eigenvalue Problems},
	pdfauthor={Qia Li, Jianmin Liao, Lixin Shen, and Na Zhang}
}
\fi

\begin{document}
\maketitle
\begin{abstract}
The Sparse Generalized Eigenvalue Problem (sGEP), a pervasive challenge in statistical learning methods including sparse principal component analysis, sparse Fisher's discriminant analysis, and sparse canonical correlation analysis, presents significant computational complexity due to its NP-hardness. The primary aim of sGEP is to derive a sparse vector approximation of the largest generalized eigenvector, effectively posing this as a sparse optimization problem. Conventional algorithms for sGEP, however, often succumb to local optima and exhibit significant dependency on initial points. This predicament necessitates a more refined approach to avoid local optima and achieve an improved solution in terms of sGEP's objective value, which we address in this paper through a novel successive two-stage method. The first stage of this method incorporates an algorithm for sGEP capable of yielding a stationary point from any initial point. The subsequent stage refines this stationary point by adjusting its support, resulting in a point with an enhanced objective value relative to the original stationary point. This support adjustment is achieved through a novel procedure we have named \textit{support alteration}. The final point derived from the second stage then serves as the initial point for the algorithm in the first stage, creating a cyclical process that continues until a predetermined stopping criterion is satisfied. We also provide a comprehensive convergence analysis of this process. Through extensive experimentation under various settings, our method has demonstrated significant improvements in the objective value of sGEP compared to existing methodologies, underscoring its potential as a valuable tool in statistical learning and optimization.

\end{abstract}



\section{Introduction}
\label{Intro}

This section reviews the concept of the sparse generalized eigenvalue problem (sGEP) and its significance in statistical learning. 


Given a matrix pair $(\mathbf{A},\mathbf{B}) \in \mathbb{S}_+^{n} \times \mathbb{S}_{++}^{n}$ (i.e., $\mathbf{A}$ is positive semi-definite and $\mathbf{B}$ is positive definite), the task of the associated generalized eigenvalue problem (GEP) is to find a pair $(\lambda,\mathbf{x}) \in \mathbb{R} \times (\mathbb{R}^n \setminus \{\mathbf{0}\})$ that satisfies $\mathbf{A} \mathbf{x}=\lambda \mathbf{B} \mathbf{x}$,  
where $\lambda$ is referred to as a generalized eigenvalue, and $\mathbf{x}$ is its associated generalized eigenvector. The generalized eigenvectors associated with the largest absolute value of generalized eigenvalues are termed as the leading generalized eigenvectors of the matrix pair ($\mathbf{A},\mathbf{B}$). 
The generalized Rayleigh quotient of a matrix pair $(\mathbf{A},\mathbf{B}) \in \mathbb{S}_+^{n} \times \mathbb{S}_{++}^{n}$ at $\mathbf{x} \neq \mathbf{0}$ is 
\begin{equation}
	R(\mathbf{x}; \mathbf{A}, \mathbf{B}):=\frac{\mathbf{\mathbf{x}}^{\top} \mathbf{A} \mathbf{\mathbf{x}}}{\mathbf{\mathbf{x}}^{\top} \mathbf{B} \mathbf{\mathbf{x}}}. \label{RayleighQ}
\end{equation}
According to the Rayleigh-Ritz theorem \cite{parlett1998symmetric}, the leading generalized eigenvector of the matrix pair ($\mathbf{A},\mathbf{B}$) is the optimal solution to the following optimization problem:
\begin{equation}
	\max \{R(\mathbf{\mathbf{x}}; \mathbf{A}, \mathbf{B}): \mathbf{\mathbf{x}} \in \mathbb{R}^n, \mathbf{x} \neq \mathbf{0}\}. \label{prob:GEP}
\end{equation}
Generalized eigenvectors of the remaining eigenvalues can be obtained via a process known as deflation \cite{cao1987deflation}, which involves solving an optimization problem similar to \eqref{prob:GEP}.



The GEP has been an active area of research for many decades, leading to the development of numerous algorithms \cite{davies2001analysis,kalantzis2020domain,lang2019efficient,meerbergen2015sylvester} to solve it. However, despite the simplicity and popularity of the GEP, a particular disadvantage is that the generalized eigenvectors typically have few zero entries, which makes the result hard to interpret \cite{song2015sparse}. The sGEP \cite{sriperumbudur2011majorization} overcomes this disadvantage by finding linear combinations that contain just a few input variables. Formally, given a matrix pair $(\mathbf{A},\mathbf{B}) \in \mathbb{S}_+^{n} \times \mathbb{S}_{++}^{n}$ and a sparsity level $s$ between $1$ and $n$, the sGEP corresponds to the following problem:
\begin{equation}
	\max \{R(\mathbf{x};\mathbf{A},\mathbf{B}): \mathbf{x}\in \mathbb{R}^n, \mathbf{x} \neq \mathbf{0}, \|\mathbf{x}\|_0 \le s\}, \label{sGEP}
\end{equation}
where $R(\mathbf{x};\mathbf{A},\mathbf{B})$ is given by \eqref{RayleighQ} and $\|\mathbf{x}\|_0$ denotes the $\ell_0$ ``norm'' of $\mathbf{x}$, which is the number of nonzero entries in $\mathbf{x}$. 

Many statistical learning models can be viewed as instances of sGEP with proper matrices $\mathbf{A}$ and $\mathbf{B}$. We list several typical examples below which will be revisited later in the section of numerical experiments. 





\textbf{Example 1: Sparse Principal Component Analysis (sPCA)} The goal of principal component analysis (PCA) is to find a projection direction that maximizes the variance of the data after being projected along this direction. 
Let $\mathbf{A}$ represent the sample covariance matrix of observations, and let $\mathbf{x}$ indicate the direction of a one-dimensional space onto which the observations will be projected. The variance of the observations projected onto $\mathbf{x}$ can be calculated as $\mathbf{x}^{\top} \mathbf{A} \mathbf{x}/\mathbf{x}^{\top}\mathbf{x}$ and the largest variance is actually the optimal value of \eqref{prob:GEP} with $\mathbf{B}=\mathbf{I}$. However, the corresponding optimal solution is typically non-sparse.  To address this issue, an approach, known as sPCA \cite{zou2006sparse}, adds a sparsity constraint to the optimization problem of PCA. Clearly, the sPCA is an instance of \eqref{sGEP} with $\mathbf{A}$ being the current $\mathbf{A}$ and $\mathbf{B}=\mathbf{I}$. The sPCA has been extensively explored in applications including big data \cite{yang2015streaming}, dimensionality reduction \cite{shi2020supervised}, manifold learning \cite{tan2019learning}, gene selection \cite{feng2019supervised} and distributed system \cite{ge2018minimax}.  





\textbf{Example 2: Sparse Fisher's Discriminant Analysis (sFDA)} Fisher's discriminant analysis (FDA) aims to determine a linear combination of features that best separates observations into two classes.   
For an optimal separation between classes, FDA aims to identify a linear combination that simultaneously maximizes between-class variance and minimizes within-class variance. FDA is designed to accomplish this by maximizing the ratio of between-class variance to within-class variance. However, the resulting linear combination in FDA typically includes all features, complicating the interpretation and utilization of the results. To address this issue, sFDA \cite{Clemmensen-Hastie-Witten:Technometrics2011} introduces a sparsity constraint on the coefficient vector, thereby limiting the number of features explicitly involved. Suppose we have two classes of observations with mean vectors $\bm{\mu}_1$, $\bm{\mu}_2$, and covariance matrices $\bm\Sigma_1$, $\bm\Sigma_2$, respectively. The formulation of sFDA is \eqref{sGEP} with 
$\mathbf{A}=(\bm{\mu}_1-\bm{\mu}_2)(\bm{\mu}_1-\bm{\mu}_2)^{\top}$ and $\mathbf{B}=\bm\Sigma_1+\bm\Sigma_2$.




\textbf{Example 3: Sparse Canonical Correlation Analysis (sCCA)}
Let $X=(X_1,\dots,X_{n_1})$ and $Y=(Y_1,\dots,Y_{n_2})$ be two random vectors from two different classes. For $\mathbf{u} \in \mathbb{R}^{n_1}$ and $\mathbf{v} \in \mathbb{R}^{n_2}$,  $U=\mathbf{u}^\top X$ is the linear combination of $X$ with the coefficient vector $\mathbf{u}$ while $V=\mathbf{v}^\top Y$ is the linear combination of $Y$ with the coefficient vector $\mathbf{v}$. The goal of canonical correlation analysis (CCA) is to seek $\mathbf{u}$ and $\mathbf{v}$ that maximize the correlation between $U$ and $V$. The corresponding CCA problem is equivalent to the following one:
\begin{equation*}\label{prob:cca1}
	\max_{\mathbf{u} \in \mathbb{R}^{n_1} , \ \mathbf{v} \in \mathbb{R}^{n_2}} \{\mathbf{u}^{\top} \bm\Sigma_{XY} \mathbf{v}: \mathbf{u}^{\top} \bm\Sigma_X \mathbf{u} =\mathbf{v}^{\top} \bm\Sigma_Y \mathbf{v}= 1\}, 
\end{equation*}
which can be further formulated as an instance of GEP \eqref{prob:GEP} with (see \cite{lee2007two})
\[
\mathbf{x}=\begin{pmatrix} \mathbf{u} \\ \mathbf{v} \end{pmatrix}, \quad \mathbf{A}=\begin{pmatrix}  \bm\Sigma_X & \bm\Sigma_{XY} \\ \bm\Sigma_{XY}^\top & \bm\Sigma_Y \end{pmatrix}, \quad \mbox{and} \quad \mathbf{B}=\begin{pmatrix} \bm\Sigma_X & \mathbf{0} \\ \mathbf{0} & \bm\Sigma_Y \end{pmatrix}, 
\]
where $\bm\Sigma_X$ and $\bm\Sigma_Y$ are the covariance matrices for $X$ and $Y$, respectively, and $\bm\Sigma_{XY}$ is the cross-covariance matrix between $X$ and $Y$. To improve interpretability, we can put sparsity constraint on the GEP formulation of CCA. The resulting optimization problem is an instance of \eqref{sGEP} with the aforementioned $\mathbf{x},\mathbf{A},\mathbf{B}$.

\textbf{Example 4: Sparse Sliced Inverse Regression (sSIR)}
Sliced inverse regression (SIR) \cite{li1991sliced} is particularly useful in high-dimensional regression analysis. Given a univariate response variable $Y$ and $n$-dimensional explanatory variables $X$, the SIR seeks $k$ linear combinations $(\mathbf{u}_1,\dots,\mathbf{u}_k)$ of $X$, $k\le n$, such that $Y$ only depends on $X$ through a $k$ dimensional subspace via a $k$ dimensional random vector $(\mathbf{u}_1^{\top}X, \dots, \mathbf{u}_k^{\top}X)$. Under some conditions \cite{li2006sparse}, the SIR is connected to $\bm\Sigma_{E(X|Y)} \mathbf{x}_j=\lambda_j \bm\Sigma_X \mathbf{x}_j$, 
where $\bm\Sigma_X$ ($\bm\Sigma_{E(X|Y)}$) is the covariance matrix for $X$ (the conditional expectation $E(X|Y)$), and $\mathbf{x}_1,\dots,\mathbf{x}_n$ are generalized eigenvectors corresponding to eigenvalues $\lambda_1\ge\dots\ge\lambda_n$, respectively. The sparse SIR \cite{li2006sparse} as a variant of the SIR was formulated as an instance of \eqref{sGEP} with $\mathbf{A}=\bm\Sigma_{E(X|Y)}$ and  $\mathbf{B}=\bm\Sigma_{X}$.

\subsection{A Brief Review of Existing Work on sGEP}\label{sec:review}




Algorithms to solve the sparse Principal Component Analysis (sPCA), an instance of sGEP illustrated in Example 1, have been extensively studied and evaluated. A basic approach to sPCA involves setting the entries of the eigenvectors of a covariance matrix to zero if their absolute values are below a certain threshold. However, this simplistic method has been proven potentially misleading in numerous aspects \cite{cadima1995loading}. Contrastingly, algorithms based on LASSO \cite{jolliffe2003modified,zou2006sparse} have shown potential in producing higher quality solutions than the simple thresholding method. Many existing methods for sPCA, as detailed in various studies \cite{journee2010generalized,luss2013conditional,shen2008sparse,sigg2008expectation,sriperumbudur2011majorization,witten2009penalized}, can be categorized as conditional gradient methods, also known as the Frank-Wolfe method \cite{frank1956algorithm}. Among these, the Truncated Power Method (TPM), independently proposed in several studies \cite{luss2013conditional,yuan2013truncated}, guarantees convergence to a stationary point of the sPCA. Compared to LASSO-based algorithms, TPM has been found to produce higher quality solutions with significantly less computational time. In high dimensional scenarios, the branch-and-bound method \cite{berk2019certifiably,moghaddam2005spectral} has shown its effectiveness in producing optimal solutions within acceptable time frames, albeit primarily in low sparsity settings. 

The introduction of an effective greedy strategy using branch-and-bound search led to further improvements \cite{moghaddam2005spectral}. PathSPCA \cite{d2008optimal} enhanced this greedy method using a different criterion. Other strategies like the coordinate-wise algorithm \cite{beck2013sparsity,beck2016sparse}, including its variants—Greedy Coordinate-Wise algorithm (GCW) and Partial Coordinate-Wise algorithm (PCW)—utilize methods such as the greedy method or PathSPCA to obtain an initial sparse solution, which is then improved by swapping one entry pair at a time, thereby enhancing the solution quality. Furthermore, several other methodologies have been used to tackle sPCA. These include sparse linear regression solver \cite{bresler2018sparse}, semi-definite programming \cite{d2005direct}, convex relaxation \cite{zhang2012sparse}, and a second-order method \cite{kuleshov2013fast}. Each of these presents unique approaches to optimizing the problem, demonstrating the breadth of methodologies available to tackle sPCA.

While a multitude of algorithms have been developed to tackle sPCA, their generalization to the sGEP case has proven challenging, resulting in fewer proposed algorithms for sGEP than for sPCA. Existing algorithms for sGEP can be broadly divided into two groups based on their approach to problem formulation. The first group addresses the penalized formulation of sGEP. Algorithms in this group, as mentioned in multiple studies \cite{song2015sparse,sriperumbudur2011majorization}, approximate the $\ell_0$ norm using surrogate functions. The resultant optimization problems are then tackled using minorization-maximization algorithms.  The second group of algorithms focuses directly on the original formulation \eqref{sGEP}. Notable among these are gradient-based algorithms, including the truncated Rayleigh flow \cite{tan2018sparse} and proximity-gradient subgradient algorithm (PGSA) with line search \cite{zhang2022first}. These algorithms were validated to produce sequences of iterations that converge to stationary points of the original formulation \cite{zhang2022first}.  Other notable methodologies include a decomposition algorithm \cite{yuan2019decomposition}, which offers high-quality solutions but is significantly more computationally intensive. A different approach incorporates a novel truncation procedure along with an eigensolver \cite{cai2020inverse} to effectively find the support set of the leading eigenvector.

Overall, while there is a range of methods to tackle sGEP, each approach has its unique challenges and advantages, emphasizing the complex and multifaceted nature of the sGEP.



\subsection{Motivations}\label{sec summary}

While the two groups of algorithms represent substantial progress towards solving sGEP, they still pose certain challenges that need to be addressed. In the case of the first group of algorithms that work with the penalized formulation of sGEP, the exact correspondence between the penalized parameter in the formulation and the sparsity of the solution is unknown prior to implementation. Additionally, setting a clear threshold to convert the approximate sparse output from the penalized formulation into genuinely sparse solutions in practical applications can be quite intricate. As for the second group of algorithms, gradient-based methods are often susceptible to becoming ensnared in local optima and exhibit a high degree of dependency on initial points. Moreover, other methods, while comprehensive, are typically time-consuming and yield solutions that offer limited improvement.

In light of these challenges, in this paper, we focus directly on the constrained formulation, aiming to derive an enhanced solution within an acceptable timeframe. Our work intends to strike a balance between efficiency and efficacy, striving to provide a viable solution for sGEP that addresses the current shortcomings in existing methods.

Coordinate-wise algorithms for sPCA offer a compelling template for achieving improved solution quality within an acceptable timeframe. The crux of these algorithms is their two-stage process. In the first stage, an initial point is derived. Following this, the second stage progressively modifies this initial point by swapping an entry pair at a time, resulting in the enhancement of the solution's quality. The strength of these algorithms lies in their ability to improve the quality of solutions without significant time consumption, as demonstrated in various experimental settings. These successful implementations serve as motivation for us to develop a new algorithm for sGEP that can successively enhance the iterates generated from the algorithm and provide insights for improving each iterate by modifying its support.

Our goal, therefore, is to apply the lessons learned from the coordinate-wise algorithms to our sGEP challenge, focusing on an iterative, support modification approach that strives for solution quality improvement within an acceptable timeframe.

\subsection{Outline}
We devise \textit{support alteration}, a procedure that targets to improve the quality of an input through adjusting its support, and outputs a potentially better initial point in terms of the objective value of the sGEP. Based on the support alteration, we devise a successive two-stage framework for sGEP, which employs an algorithm to obtain an initial point in the first stage, and employs the support alteration to obtain a support-altered solution in the second stage. The resulting solution from the second stage will serve as the initial point of the algorithm in the first stage. This process will be repeated until some stopping criteria are met. This algorithm provides a significant numerical improvement compared with the existing algorithms, as demonstrated in Section \ref{sec experiment}.



In Section~\ref{sec:existence}, we first demonstrate the existence of an optimal solution to the sGEP. We then discuss the NP-hardness of the sGEP.

In Section \ref{sec framework}, we propose a successive two-stage framework with a pre-given step size for support alteration. Our framework consists of two stages in each iteration. In the first stage of the framework, it performs a gradient-based framework for the sGEP. Its output is further refined through a support alteration in the second stage. Our framework performs these two stages successively until a predetermined stopping criterion is met. In particular, we provide a closed-form solution for a key one-dimensional optimization problem arsing in the second stage.  

In Section \ref{sec pair number}, our primary focus is on self-adaptively determining the number of swapped pairs, represented by $r$. Initially, we discuss that $r$ can be viewed as a step size in support alterations when $\ell_0$ norm is used as the metric. Drawing inspiration from the backtracking technique frequently used in numerical algorithms, we propose a self-adaptive strategy for determining $r$ in each iteration. Subsequently, incorporating this self-adaptive strategy and the diminishing strategy for choosing step size, we introduce a holistic successive two-stage algorithm with self-adaptive number of swapped pairs for the sGEP. Additionally, this section provides a theorem on the convergence properties of our algorithm and concludes with a comparative analysis against coordinate-wise algorithms.


In Section \ref{sec experiment}, comprehensive experiments on several sGEP instances are conducted. The experiments demonstrate that our successive two-stage algorithm provides best solutions on average when compared to state-of-the-art methods. The conclusion of this paper is drawn in Section~\ref{sec:conclusions}.






\section{Existence and NP-Hardness} \label{sec:existence}
In this section, we study the existence of optimal solutions to sGEP and the NP-hardness of the problem. Before that, we first present our preferred notations used in this paper.   
Throughout this paper, $\mathbb{R}$, $\mathbb{R}^{n}$, and $\mathbb{N}$ denote the sets of the real numbers, $n$-dimensional real vectors, and all natural numbers, respectively. For an $r \in \mathbb{N}$, we define $\mathbb{N}_r=\{0, 1, \ldots, r\}$ and $\mathbb{N}_r^+=\{1, \ldots, r\}$. We write $\mathbb{S}_+^{n}$ ($\mathbb{S}_{++}^{n}$) to  denote the set of symmetric positive semidefinite (positive definite) matrices of size $n$.  Boldface uppercase letters denote matrices, and lowercase letters denote vectors.  $[\mathbf{x}]_i$ denotes the $i$-th component of $\mathbf{x}$. The entry in the $i$-th row and $j$-th column of matrix $\mathbf{A}$ is denoted by $[\mathbf{A}]_{i,j}$. The $i$-th row and the $j$-th column of $\mathbf{A}$ are denoted by $[\mathbf{A}]_{i,:}$ and $[\mathbf{A}]_{:,j}$ respectively. $\mathbf{0}$ and $\mathbf{I}$ denote the zero vector and the identity matrix of appropriate dimension, respectively. $\mathbf{e}_i$ denotes the $i$-th standard basis vector. The generalized Rayleigh quotient of the matrix pair ($\mathbf{A}$,$\mathbf{B}$) is $R(\mathbf{x};\mathbf{A},\mathbf{B})$ given in \eqref{RayleighQ}. When $\mathbf{A}$ and $\mathbf{B}$ are clear from the context, we will use $R(\mathbf{x})$ to replace $R(\mathbf{x};\mathbf{A},\mathbf{B})$ for simplicity.  For $\mathbf{x} \in \mathbb{R}^n$, $\|\mathbf{x}\|_0$ denotes the number of nonzero entries of $\mathbf{x}$, and $\|\mathbf{x}\|_2$ denotes the $\ell_2$ norm of $\mathbf{x}$. $\mathcal{S}({\mathbf{x}}):=\{i:[\mathbf{x}]_i \neq 0\}$ denotes the support of $\mathbf{x}$, i.e., the index set of the nonzero entries of $\mathbf{x}$. $\mathcal{Z}({\mathbf{x}}):=\{i:[\mathbf{x}]_i = 0\}$ is the index set of the zero entries of $\mathbf{x}$. For a set $S$, we write $|S|$ to represent its cardinality.

The existence of optimal solutions to sGEP is guarantted by the following theorem. 
\begin{theorem}\label{existence thm}
The set of optimal solutions to sGEP~\eqref{sGEP} is nonempty.
\end{theorem}
\begin{proof}\ \ Notice that both functions $R(\cdot; \mathbf{A},\mathbf{B})$ and $\|\cdot\|_0$ are homogeneous with degree zero, that is, for any nonzero scalar $\alpha$ and nonzero vector $\mathbf{x}$, 
$$
R(\alpha\mathbf{x}; \mathbf{A},\mathbf{B})=R(\mathbf{x}; \mathbf{A},\mathbf{B}) \quad \mbox{and} \quad 
\|\alpha \mathbf{x}\|_0=\|\mathbf{x}\|_0.
$$
Let $X^*$ be the set of optimal solutions of \eqref{sGEP} and let $\tilde{X}^*$ be the set of optimal solutions of the following optimization problem 
\begin{equation}\label{sGEP-1}
    \max\{R(\mathbf{x}; \mathbf{A},\mathbf{B}): \mathbf{x}\in \mathbb{R}^n, \mathbf{x}^\top\mathbf{B}\mathbf{x}=1,  \|\mathbf{x}\|_0\le s\}.
\end{equation} 
Then, we should have $X^* = \cup_{\alpha \neq 0} \alpha \tilde{X}^*$. Therefore, to show $X^*$ nonempty, it is sufficient to show that the set $\tilde{X}^*$ is nonempty.  

Under the constraint of $\mathbf{x}^\top\mathbf{B}\mathbf{x}=1$, the function $R(\mathbf{x}; \mathbf{A},\mathbf{B})=\mathbf{x}^\top \mathbf{A} \mathbf{x}$ is continuous with respect to $\mathbf{x}$. The set of $\{\mathbf{x}: \mathbf{x}^\top\mathbf{B}\mathbf{x}=1\}$ is compact due to the positive definiteness of $\mathbf{B}$, and the set of $\{\mathbf{x}: \|\mathbf{x}\|_0\le s\}$ is closed from the lower-semicontinuousness of the $\ell_0$ norm. Therefore, the intersection of these two sets is compact. Thus, the optimal solutions to \eqref{sGEP-1} are achievable. This completes the proof of this theorem. 
\end{proof}


The task of solving an sGEP consists of two parts: finding an optimal support, which is the support of any optimal solution, and computing the largest generalized eigenvalue of the submatrices corresponding to that support. The computational complexity for the second part is  $O(s^3)$, where $s$ is the sparsity level of \eqref{sGEP}. The complexity of solving sGEP is dominated by the first part, which is NP-hard. The number of possible choices for the optimal support is 
${n \choose s}=\frac{n!}{s!(n-s)!}$. Under the assumptions that the optimal support is unique for an sGEP and that $s$ is much smaller than $n$, the effects to the possible choices of the optimal support by increasing the sparsity $s$ and the dimension $n$ each by one reflect in the following relations
\[
{n\choose s+1}/{n\choose s}=\frac{n-s}{s+1}\approx \frac{n}{s} \quad \mbox{and}\quad
{n+1\choose s}/{n\choose s}=\frac{n+1}{n+1-s}\approx 1
\]
for a large $n$. The first relation in the above says that the complexity of finding the optimal support increases rapidly with respect to the sparsity $s$ while the second one indicates that the complexity of finding the optimal support increases much slower with respect to the dimension $n$. Succinctly, the NP-hardness of sGEP is mainly contributed by the sparsity of the problem rather than its dimension. We shall expect that an algorithm solving sGEP exactly would consume an unacceptable amount of time even under a setting with moderate sparsity and dimension. Hence, it is very interesting to develop efficient algorithms for solving sGEP approximately. 

\section{The Successive Two-Stage Framework for sGEP} \label{sec framework}
In this section, we will propose a successive two-stage framework for sGEP that is described in  Algorithm \ref{alg:generalframework}. There are two main stages in the pseudo-code of this algorithm. At the $t$-th iteration, Stage-1 obtains $\mathbf{x}^{(t)}$ from ${\hat{\mathbf{x}}}^{(t-1)}$, the output of Stage-2 in the previous iteration, using an algorithm for sGEP; Stage-2 obtains $\hat{\mathbf{x}}^{(t)}$ based on $\mathbf{x}^{(t)}$ using a proposed support alteration. The two-stage framework performs successively until convergence is achieved. The general requirements for two stages are outlined as follows: 

\begin{algorithm}[htb] 
	\caption{Successive two-stage framework for sGEP}	
	\begin{algorithmic}
		\STATE {\bfseries Input:} $\mathbf{A} \in \mathbb{S}_+^n$, $\mathbf{B} \in \mathbb{S}_{++}^n$, and sparsity level $s$.
		\STATE {\bfseries Initialization:} $t:=0$. Find $\hat{\mathbf{x}}^{(0)}\neq 0$ such that $\|\hat{\mathbf{x}}^{(0)}\|_0\le s$.
		\REPEAT
            \STATE $t:=t+1$.
		\STATE Stage-1: Acquire $\mathbf{x}^{(t)}$ from ${\hat{\mathbf{x}}}^{(t-1)}$ through an algorithm for sGEP. 
		\STATE Stage-2: Obtain ${\hat{\mathbf{x}}}^{(t)}$ from $\mathbf{x}^{(t)}$ via a support alteration.
		\UNTIL {Convergence}
		\STATE {\bfseries Output:} $\mathbf{x}^{(t)}$
	\end{algorithmic}
 \label{alg:generalframework}
\end{algorithm}

\textit{Stage-1.} This stage employs a specific algorithm to solve problem \eqref{sGEP}, beginning with an initial point. The chosen iterative algorithm should satisfy the following criteria:
\begin{itemize}
    \item \textit{Monotonicity:} The objective value of the generated output, $R(\mathbf{x}^{(t)})$, should be no less than $R({\hat{\mathbf{x}}}^{(t-1)})$, the objective value of the initial point ${\hat{\mathbf{x}}}^{(t-1)}$.
    \item \textit{Global convergence:} The employed algorithm converges to some points with certain optimality conditions for \eqref{sGEP}, such as stationarity.
    \item \textit{Efficiency:} { The employed algorithm offers rapid convergence and has a low computational complexity.}
\end{itemize}




\textit{Stage-2.} The goal of the second stage is to generate a vector ${\hat{\mathbf{x}}}^{(t)}$ based on ${\mathbf{x}}^{(t)}$. This ${\hat{\mathbf{x}}}^{(t)}$ will be used as the initial point of the algorithm for the first stage in the next iteration, aiming to generate a superior ${\mathbf{x}}^{(t+1)}$ than ${\mathbf{x}}^{(t)}$ in the sense that $R({\mathbf{x}}^{(t+1)})>R({\mathbf{x}}^{(t)})$.
The key of achieving this goal lies in selecting an ${\hat{\mathbf{x}}}^{(t)}$ that satisfies two primary conditions:
\begin{itemize}
    \item ${\hat{\mathbf{x}}}^{(t)}$ must be feasible.

    \item ${\hat{\mathbf{x}}}^{(t)}$ should not be proximate to $\mathbf{x}^{(t)}$ as this $\mathbf{x}^{(t)}$ has already been a local optimum. The following discussion will show that $\mathcal{S}({\hat{\mathbf{x}}}^{(t)})\ne \mathcal{S}(\mathbf{x}^{(t)})$ guarantees this condition.
\end{itemize}

Delving into these requirements, we consider any arbitrary $\mathbf{x}$ such that $\|\mathbf{x}\|_0\le s$ and $\mathcal{S}(\mathbf{x}) \ne \mathcal{S}(\mathbf{x}^{(t)})$. Immediately, we have that if $\|\mathbf{x}^{(t)}\|_0=s$ there holds
$$
\|\mathbf{x}-\mathbf{x}^{(t)}\|_2 \ge \min_{j \in \mathcal{S}(\mathbf{x}^{(t)})}|[\mathbf{x}^{(t)}]_j|>0.
$$ 
This means that any feasible $\mathbf{x} \in \mathbb{R}^n$ having a different support from $\mathbf{x}^{(t)}$  must be at least a certain distance away from $\mathbf{x}^{(t)}$, therefore, satisfying the second requirement. To find a ${\hat{\mathbf{x}}}^{(t)} \in \mathbb{R}^n$ such that the support of  ${\hat{\mathbf{x}}}^{(t)}$ is different from that of $\mathbf{x}^{(t)}$, we will propose a strategy called \textit{support alteration}. This approach generates a new initial point by suitably altering the support of $\mathbf{x}^{(t)}$. The intention is to increase the likelihood of the first stage  delivering a higher objective value in the subsequent iteration. 

The rest of this section will provide details on the two stages in Algorithm~\ref{alg:generalframework}. 

\subsection{Gradient-Based Algorithms}  Stage 1 of Algorithm~\ref{alg:generalframework} requires a fast and efficient solver for sGEP. 
Any algorithm meeting the criteria listed in the above can be chosen as a solver for Stage-1 of Algorithm \ref{alg:generalframework}. We will use gradient-based algorithms in Stage-1, owing to their simplicity in implementation and their theoretical guarantees of convergence. Specifically speaking, we adopt TPM \cite{yuan2013truncated} for sPCA case ($\mathbf{B}=\mathbf{I}$),  and PGSA with monotone line search (PGSA\_ML) \cite{zhang2022first} or rifle (Truncated Rayleigh flow method) \cite{tan2018sparse} for a general sGEP case as the solvers used in the first stage. The pseudo code of PGSA\_ML is stated in Algorithm \ref{alg:PGSA}, where $\text{Truncate}(\mathbf{x},s)$ denotes an operator that outputs a vector preserving the largest $s$ entries in absolute of $\mathbf{x}$ and has all remaining entries being zero. Actually, TPM is a special case of PGSA\_ML when $\mathbf{B}=\mathbf{I}$, $a=0$ and $\alpha_k=\frac{1}{2\|\mathbf{B}\|_2^2}=\frac{1}{2}$ for $k\in\mathbb{N}$. Rifle is also a special case of PGSA\_ML as $\alpha_k$ is fixed to be a number in $(0, \frac{1}{2\|\mathbf{B}\|_2^2})$ for all $k\in\mathbb{N}$. Here, as suggested in \cite{zhang2022first}, we choose the initial $\alpha_k$ for PGSA\_ML in the following formula
\begin{equation*}
    \alpha_k=
    \begin{cases}
        \max\{\underline{\alpha},\min\{\overline{\alpha},\frac{\|\Delta \mathbf{x}\|_2^2}{|\langle\Delta \mathbf{x},\Delta \mathbf{B} \mathbf{x}\rangle|}\}\}, &\text{  if } \langle\Delta \mathbf{x},\Delta \mathbf{B} \mathbf{x}\rangle\ne 0,\\
        \overline{\alpha}, & \text{  else,}
    \end{cases},
\end{equation*}
where $\Delta \mathbf{x}:=\mathbf{x}^{(k-1)}-\mathbf{x}^{(k-2)}$ and $\Delta \mathbf{B}\mathbf{x}:=2\mathbf{B}(\mathbf{x}^{(k-1)}-\mathbf{x}^{(k-2)})$.


\begin{algorithm}[htb] 
	\caption[PGSA with monotone line search (PGSA\_ML)]{PGSA with Monotone Line search (PGSA\_ML)}	
	\begin{algorithmic}
        \label{alg:PGSA}
        \STATE {\bfseries Input:} $\mathbf{A} \in \mathbb{S}_+^n$, $\mathbf{B} \in \mathbb{S}_{++}^n$, sparsity level $s$. 
        \STATE {\bfseries Initialization:} An initial point $\mathbf{x}^{(0)}\neq 0$ such that $\|\mathbf{x}^{(0)}\|_0\le s$ and $\mathbf{x}^{(0)\top} \mathbf{A}\mathbf{x}^{(0)}\ne 0$, $a\ge0$, $\eta\in (0,1)$, $0<\underline{\alpha}<\overline{\alpha}$; set  $k:=0$.
        \REPEAT
            \STATE $k := k+1$ and choose $\alpha_k\in [\underline{\alpha},\overline{\alpha}]$.
            \REPEAT
                \STATE $\mathbf{\tilde x}:=\text{Truncate}\left(\mathbf{x}^{(k-1)}+2\alpha_{k}(-\mathbf{B}\mathbf{x}^{(k-1)}+\frac{\mathbf{A}\mathbf{x}^{(k-1)}}{R(\mathbf{x}^{(k-1)})}),s\right)$.
                \STATE $\mathbf{\tilde x}:=\mathbf{\tilde x}/\|\mathbf{\tilde x}\|_2$.
                \STATE $\alpha_k:=\alpha_k \eta$.
            \UNTIL {$\frac{1}{R(\mathbf{\tilde x})}\le \frac{1}{R(\mathbf{x}^{(k-1)})}-\frac{a}{2}\|\mathbf{\tilde x}-\mathbf{x}^{(k-1)}\|_2^2$}.
            \STATE $\mathbf{x}^{(k)} := \mathbf{\tilde{x}}$.
        \UNTIL {convergence}
        \STATE {\bfseries Output:} $\mathbf{x}^{(k)}$
	\end{algorithmic}
\end{algorithm}

\subsection{Support Alteration}
This subsection will focus on \textit{support alteration}, a procedure which finds a potentially improved initial point for the gradient-based algorithm employed in Stage 1. Suppose $\mathbf{x} \in \mathbb{R}^n$ is a nonzero vector that we obtained in Stage 1. To alter the support of $\mathbf{x}$ to obtain a new solution as a potentially better initial point for the gradient-based algorithm, support alteration finds a vector $\hat{\mathbf{x}}$ having the following form
\begin{equation}\label{eq:x_hat}
    \hat{\mathbf{x}}=\mathbf{x} + \sum_{t=1}^{r} (-[\mathbf{x}]_{j_t} \mathbf{e}_{j_t} + \alpha_t \mathbf{e}_{i_t})
\end{equation}
such that $\mathcal{S}(\hat{\mathbf{x}})$ potentially consists of $r$ elements from $\mathcal{Z}(\mathbf{x})$ and $(\|\mathbf{x}\|_0-r)$ elements from $\mathcal{S}(\mathbf{x})$ corresponding to the largest Rayleigh quotient. Figure \ref{fig:x_hat} illustrates the process of acquiring $\hat{\mathbf{x}}$ from $\mathbf{x}$ by altering its support.
\begin{figure}[htbp]
    \centering
    \includegraphics[scale=0.7,trim=40 620 80 50]{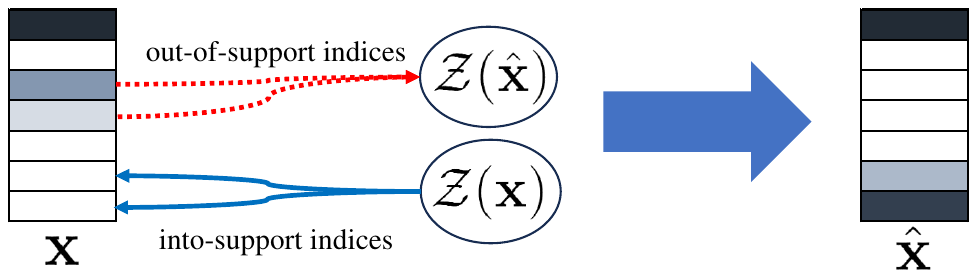}
    \caption{Illustration of acquiring $\hat{\mathbf{x}}$ from $\mathbf{x}$ in the case $r=2$. The array labelled by $\mathbf{x}$ is the vector to be altered while the array labelled by $\hat{\mathbf{x}}$ is the altered vector of $\mathbf{x}$. Each rectangle in an array represents an entry of the corresponding vector. Colored rectangles represent nonzero entries while white rectangles represent zero entries. The vector $\hat{\mathbf{x}}$ is formed from  $\mathbf{x}$ by moving $r$ indices in $\mathcal{S}(\mathbf{x})$ into $\mathcal{Z}(\hat{\mathbf{x}})$ and the same number of indices in  $\mathcal{Z}(\mathbf{x})$ are moved into $\mathcal{S}(\hat{\mathbf{x}})$. }
    \label{fig:x_hat}
\end{figure}
We call the parameter $r$ the number of swapped pairs or the number of indices to be altered. Specifically, $\{j_1,\ldots,j_r\} \subset \mathcal{S}(\mathbf{x})$ represents the set of \textit{out-of-support} indices, and $\{i_1,\ldots,i_r\} \subset \mathcal{Z}(\mathbf{x})$ denotes the set of \textit{into-support} indices.
For this vector $\mathbf{x}$ with $\|\mathbf{x}\|_0 \le s$, we define the feasible set for a fixed $r \in \mathbb{N}_{\|\mathbf{x}\|_0}$ that 
$$
\mathcal{P}_r(\mathbf{x}):=\{\hat{\mathbf{x}}=\mathbf{x} + \sum_{t=1}^{r} (-[\mathbf{x}]_{j_t} \mathbf{e}_{j_t} + \alpha_t \mathbf{e}_{i_t}): (j_t,i_t) \in \mathcal{S}(\mathbf{x})\times \mathcal{Z}(\mathbf{x}),\alpha_t\in \mathbb R\}.
$$
The set of all $\hat{\mathbf{x}}$ having the form \eqref{eq:x_hat} is
$$
\mathcal{P}(\mathbf{x}):=\cup_{r=0}^ {\|\mathbf{x}\|_0}\mathcal{P}_r(\mathbf{x}).
$$
Among all $\hat{\mathbf{x}}$ in $\mathcal{P}(\mathbf{x})$, we desire to find the optimal one, denoted by $\hat{\mathbf{x}}^*$, as the solution of the following optimization problem
\begin{equation}\label{prob:x_hat}
    \hat{\mathbf{x}}^*:=\arg\max\{R(\hat{\mathbf{x}}):\hat{\mathbf{x}}\in\mathcal{P}(\mathbf{x})\}.
\end{equation}
Under some conditions, the relationship between \eqref{prob:x_hat} and \eqref{sGEP} is discussed as follows.
\begin{Proposition}\label{prop:optimal}
Let $\mathbf{x}^*$ be an optimal solution to \eqref{sGEP}. For a given $\mathbf{x}$, let $S:={\mathcal{S}(\mathbf{x})\cap \mathcal{S}(\mathbf{x}^*)}$. If  $[\mathbf{x}]_S=a [\mathbf{x}^*]_S$ for some $a\in \mathbb{R}\setminus \{0\}$, then there exists a vector $\hat{\mathbf{x}}^* \in \mathcal{P}_r(\mathbf{x})$ with $r=\|\mathbf{x}\|_0-|S|$  such that $\hat{\mathbf{x}}^*$ is a solution to \eqref{prob:x_hat}.  
\end{Proposition}
\begin{proof}\ \  We write   
$$
\hat{\mathbf{x}}^*=\mathbf{x}-\sum_{t \in \mathcal{S}(\mathbf{x})\setminus S}[\mathbf{x}]_t \mathbf{e}_t + a \sum_{t \in \mathcal{S}(\mathbf{x}^*)\setminus S}[\mathbf{x}^*]_t \mathbf{e}_t,
$$
which is in $\mathcal{P}_r(\mathbf{x})$. From the condition $[\mathbf{x}]_S=a [\mathbf{x}^*]_S$, we have  $\hat{\mathbf{x}}^*=a \mathbf{x}^*$. This completes the proof from the homogeneity of $R(\cdot)$. 
\end{proof}

To obtain an optimal $\hat{\mathbf{x}}^*$ of \eqref{prob:x_hat}, we need to obtain an optimal $r$, $(j_t,i_t)$, and $\alpha_t$. It looks unclear how to optimize those parameters as a whole. Hence, we divide those parameters to conquer the problem. The following part of this section supposes $r$ is pre-given, and focuses on optimize $(j_t,i_t)$, and $\alpha_t$. That is, support alteration with the pre-given $r$ solves the following problem
\begin{equation}\label{prob:x_hat_r}
    \hat{\mathbf{x}}^*_r:=\arg\max\{R(\hat{\mathbf{x}}):\hat{\mathbf{x}}\in\mathcal{P}_r(\mathbf{x})\}.
\end{equation}
To tackle \eqref{prob:x_hat_r}, we 
simplify the selection of $\alpha_1,\dots,\alpha_r$ by assuming that each $\alpha_t$, where $t \in \mathbb{N}_r^+$, solely depends on its corresponding index pair, $(j_t,i_t)$. That is, for each index pair $(j,i)$, the corresponding $\alpha$ is an arbitrary maximum of the following univariate problem:
\begin{equation}
	\alpha_{(j,i),\mathbf x}^*:=\arg\max\{\tilde{R}_{(j,i),\mathbf x}(\alpha): \ \alpha \in \mathbb{R}\}, \label{bestalpha}
\end{equation}
where 
\begin{equation}
	\tilde{R}_{(j,i),\mathbf x}(\alpha):=R(\mathbf x + (-[\mathbf x]_j \mathbf{\mathbf e}_{j} + \alpha \mathbf{e}_{i})). \label{def:tidleR}
\end{equation} 
The following proposition confirms that maxima set $\alpha^*_{(j,i),\mathbf{x}}$ in \eqref{bestalpha} has a closed-form expression. 
\begin{theorem}\label{thm sol}
Let $(\mathbf{A},\mathbf{B})$ be a matrix pair in $\mathbb{S}_+^{n} \times \mathbb{S}_{++}^{n}$. For a nonzero vector $\mathbf{x} \in \mathbb{R}^n$ and an index pair $(j,i) \in \mathbb{N}_n^+ \times \mathcal{Z}({\mathbf{x}})$ with $j\ne i$, let $\mathbf{y}:= \mathbf{x} - [\mathbf{x}]_j \mathbf{e}_{j}$. Then, the following statements for the set $\alpha^*_{(j,i),\mathbf{x}}$ defined in \eqref{bestalpha} hold: 
\begin{itemize}
\item[(i)] If $\|\mathbf y\|_0 = 0$, then $\alpha^*_{(j,i),\mathbf{x}}=\mathbb{R} \setminus \{0\}$.

\item[(ii)]   If $\|\mathbf y\|_0 \ge 1$, then
\begin{align*}
\alpha^*_{(j,i),\mathbf{x}}=&
\left\{\begin{array}{rl}
\mathbb{R}, & \mbox{ if $|[\mathbf{Q}]_{:,(1,2)}|=|[\mathbf{Q}]_{:,(1,3)}|=0$,}\\
\{-\frac{|[\mathbf{Q}]_{:,(2,3)}|}{|[\mathbf{Q}]_{:,(1,3)}|}\},   &  \mbox{ if $|[\mathbf{Q}]_{:,(1,2)}|=0$ and $|[\mathbf{Q}]_{:,(1,3)}|<0$,}\\
  \{+\infty, -\infty\},   & \mbox{ if  $|[\mathbf{Q}]_{:,(1,2)}|=0$ and $|[\mathbf{Q}]_{:,(1,3)}|>0$,}\\
\{\alpha^-\},   &  \mbox{ if $|[\mathbf{Q}]_{:,(1,2)}|\ne0$,}
\end{array}
\right.
\end{align*}
where 
$$
\mathbf{Q}:=\begin{bmatrix}
     [\mathbf{A}]_{i,i} & [\mathbf{A}\mathbf{y}]_{i}&\mathbf{y}^\top \mathbf{A}\mathbf{y} \\
     [\mathbf{B}]_{i,i} & [\mathbf{B}\mathbf{y}]_{i}&\mathbf{y}^\top \mathbf{B}\mathbf{y}
\end{bmatrix}
$$
with $[\mathbf{Q}]_{:,(1,2)}$, $[\mathbf{Q}]_{:,(1,3)}$, and $[\mathbf{Q}]_{:,(2,3)}$ are the $2 \times 2$ matrices formed by the first two columns, the first and the third columns, and the last two columns of $\mathbf{Q}$, respectively; and 
\[\alpha^-:=\frac{-|[\mathbf{Q}]_{:,(1,3)}|-\sqrt{|[\mathbf{Q}]_{:,(1,3)}|^2-4|[\mathbf{Q}]_{:,(1,2)}|\cdot |[\mathbf{Q}]_{:,(2,3)}|}}{2|[\mathbf{Q}]_{:,(1,2)}|}.\]
\end{itemize}
\end{theorem}

\begin{proof}
Part (i): The condition $\|\mathbf{y}\|_0=0$ implies
$$\tilde{R}_{(j,i),\mathbf{x}}(\alpha)=R(\mathbf{y} + \alpha \mathbf{e}_i)=R(\alpha \mathbf{e}_i)= \frac{[\mathbf{A}]_{i,i}}{[\mathbf{B}]_{i,i}}.
$$
Hence,  $\alpha^*_{(j,i),\mathbf{x}}=\mathbb{R} \setminus \{0\}$. 

Part (ii): The condition $\|\mathbf{y}\|_0 \ge 1$ and $i \in \mathcal{Z}({\mathbf{x}}) \subseteq \mathcal{Z}({\mathbf{y}})$ imply $\mathbf{y}+\alpha \mathbf{e}_i \neq \mathbf{0}$. We have $\tilde{R}_{(j,i),\mathbf{x}}(\alpha)=R(\mathbf{y}+\alpha \mathbf{e}_i)$, which can be further written as 
$$
\tilde{R}_{(j,i),\mathbf{x}}(\alpha)=\frac{(\mathbf{y}+\alpha \mathbf{e}_i)^\top \mathbf{A} (\mathbf{y}+\alpha \mathbf{e}_i)}{(\mathbf{y}+\alpha \mathbf{e}_i)^\top \mathbf{B} (\mathbf{y}+\alpha \mathbf{e}_i)}
=\frac{[\mathbf{A}]_{i,i}\alpha^2+2[\mathbf{A} \mathbf{y}]_i \alpha +\mathbf{y}^\top \mathbf{A}\mathbf{y}}{[\mathbf{B}]_{i,i}\alpha^2+2[\mathbf{B} \mathbf{y}]_i \alpha +\mathbf{y}^\top \mathbf{B}\mathbf{y}},
$$
whose domain is the entire real line.  
Taking the derivative of $\tilde{R}_{(j,i),\mathbf{x}}(\alpha)$ with respect to $\alpha$, after some simplifications, it yields 
$$
\frac{d}{d \alpha}\tilde{R}_{(j,i),\mathbf{x}}(\alpha) = \frac{2(|[\mathbf{Q}]_{:,(1,2)}|\alpha^2+|[\mathbf{Q}]_{:,(1,3)}|\alpha+|[\mathbf{Q}]_{:,(2,3)}|)}{([\mathbf{B}]_{i,i}\alpha^2+2[\mathbf{B} \mathbf{y}]_i \alpha +\mathbf{y}^\top \mathbf{B}\mathbf{y})^2}.
$$
To determine which critical points of $\tilde{R}_{(j,i),\mathbf{x}}$ are maximum, we consider three cases according to the values of $|[\mathbf{Q}]_{:,(1,2)}|$ and $|[\mathbf{Q}]_{:,(1,3)}|$.  The following identity is also useful in our discussion
\begin{equation}\label{eq:R-infty}
\lim_{\alpha \rightarrow - \infty}\tilde{R}_{(j,i),\mathbf{x}}(\alpha)=\lim_{\alpha \rightarrow + \infty}\tilde{R}_{(j,i),\mathbf{x}}(\alpha)=\frac{[\mathbf{A}]_{i,i}}{[\mathbf{B}]_{i,i}}.
\end{equation}

Case 1: $|[\mathbf{Q}]_{:,(1,2)}|=|[\mathbf{Q}]_{:,(1,3)}|=0$. In this case, we know that the first row of $\mathbf{Q}$ is a multiple, say $q$, of the second row of $\mathbf{Q}$. This implies that $\tilde{R}_{(j,i),\mathbf{x}}(\alpha)=q$. We conclude that  $\alpha^*_{(j,i),\mathbf{x}} = \mathbb{R}$.

Case 2: $|[\mathbf{Q}]_{:,(1,2)}|=0$ and $|[\mathbf{Q}]_{:,(1,3)}| \neq 0$. There are two possibilities: If $|[\mathbf{Q}]_{:,(1,3)}|<0$, then $\tilde{R}_{(j,i),\mathbf{x}}(\alpha)$ is increasing on the interval $(-\infty, -\frac{|[\mathbf{Q}]_{:,(2,3)}|}{|[\mathbf{Q}]_{:,(1,3)}|})$, and decreasing on the interval $(-\frac{|[\mathbf{Q}]_{:,(2,3)}|}{|[\mathbf{Q}]_{:,(1,3)}|}, \infty)$;   If $|[\mathbf{Q}]_{:,(1,3)}|>0$, then $\tilde{R}_{(j,i),\mathbf{x}}(\alpha)$ is decreasing on the interval $(-\infty, -\frac{|[\mathbf{Q}]_{:,(2,3)}|}{|[\mathbf{Q}]_{:,(1,3)}|})$ and increasing on the interval $(-\frac{|[\mathbf{Q}]_{:,(2,3)}|}{|[\mathbf{Q}]_{:,(1,3)}|}, \infty)$. From \eqref{eq:R-infty}, we obtain 
\[
\alpha^*_{(j,i),\mathbf{x}} =\left\{\begin{array}{rl}
  \{-\frac{|[\mathbf{Q}]_{:,(2,3)}|}{|[\mathbf{Q}]_{:,(1,3)}|}\},   &  \mbox{ if $|[\mathbf{Q}]_{:,(1,3)}|<0$,}\\
  \{+\infty, -\infty\},   & \mbox{ if $|[\mathbf{Q}]_{:,(1,3)}|>0$.}
\end{array}
\right.
\]


Case 3: $|[\mathbf{Q}]_{:,(1,2)}|\neq 0$. First, we show the discriminant $\Delta:=|[\mathbf{Q}]_{:,(1,3)}|^2-4|[\mathbf{Q}]_{:,(1,2)}|\cdot |[\mathbf{Q}]_{:,(2,3)}|$ of the quadratic polynomial $|[\mathbf{Q}]_{:,(1,2)}|\alpha^2+|[\mathbf{Q}]_{:,(1,3)}|\alpha+|[\mathbf{Q}]_{:,(2,3)}|$ is always positive. Assume to the contrary that the discriminant is non-positive, then $\tilde{R}_{(j,i),\mathbf{x}}$ is strictly monotonic on the real line. However, \eqref{eq:R-infty}
contradicts the strict monotonicity property of $\tilde{R}_{(j,i),\mathbf{x}}$. Therefore, the polynomial  $|[\mathbf{Q}]_{:,(1,2)}|\alpha^2+|[\mathbf{Q}]_{:,(1,3)}|\alpha+|[\mathbf{Q}]_{:,(2,3)}|$ has two distinct roots $\alpha^-$ and $\alpha^+$ as follows:
$$
\alpha^-=\frac{-|[\mathbf{Q}]_{:,(1,3)}|-\sqrt{\Delta}}{2|[\mathbf{Q}]_{:,(1,2)}|},\ \alpha^+=\frac{-|[\mathbf{Q}]_{:,(1,3)}|+\sqrt{\Delta}}{2|[\mathbf{Q}]_{:,(1,2)}|}.
$$
If $|[\mathbf{Q}]_{:,(1,2)}|>0$, then $\alpha^-<\alpha^+$. In this scenario, the function $\tilde{R}_{(j,i),\mathbf{x}}$ is increasing on $(-\infty, \alpha^-)$ and $(\alpha^+, \infty)$, and decreasing on $(\alpha^-,\alpha^+)$. On the contrary, if $|[\mathbf{Q}]_{:,(1,2)}|<0$, then $\alpha^->\alpha^+$. The function $\tilde{R}_{(j,i),\mathbf{x}}$ is decreasing on $(-\infty, \alpha^-)$ and $(\alpha^+, \infty)$, and increasing on $(\alpha^-,\alpha^+)$. Together with \eqref{eq:R-infty}, we conclude that $\alpha^*_{(j,i),\mathbf{x}} =\{\alpha^-\}$. 
\end{proof}

\begin{remark}
For the situation of $\alpha^*_{(j,i),\mathbf{x}}=\{+\infty, -\infty\}$ in Theorem~\ref{thm sol}, it means that the maximum value of $\tilde{R}_{(j,i),\mathbf{x}}$ is unattainable, but from \eqref{eq:R-infty} we do have 
$$
\lim_{\alpha \rightarrow \pm \infty}\tilde{R}_{(j,i),\mathbf{x}}(\alpha)=R(a\mathbf{e}_i),
$$  
for any $a\ne 0$.
\end{remark}

Now, we define $M_{(j,i),\mathbf x}$ as the maximum value of $\tilde{R}_{(j,i),\mathbf x}$ on $\mathbb{R}$ for 
$(j,i) \in \mathcal{S}({\mathbf{x}}) \times \mathcal{Z}({\mathbf{x}})$, that is, 
$$
M_{(j,i),\mathbf{x}}=\tilde{R}_{(j,i),\mathbf{x}}(\alpha^*_{(j,i),\mathbf{x}}),
$$ 
where $\alpha^*_{(j,i),\mathbf{x}}$ is understood as any arbitrary element in the maximum set given in \eqref{bestalpha}.  
Assuming each $\alpha_t$ solely depends on its corresponding index implies $\alpha_t$ is determined by $(j_t,i_t)$, leaving $(j_t,i_t)$ the only optimization variables. This allows us to employ a greedy method to tackle \eqref{prob:x_hat_r}. We present \textit{greedy support alteration} with pseudo-code summarized in Algorithm \ref{alg:greedy}. 
\begin{algorithm}[htbp]
	\caption{Greedy support alteration}	
	\begin{algorithmic}
		\label{alg:greedy}
		\STATE {\bfseries Input:} a solution $\mathbf x\in \mathbb{R}^n$ and $r$ index pairs to be swapped with $r \le \min\{\|\mathbf{x}\|_0,n-\|\mathbf{x}\|_0\}$.
		\STATE  {\bfseries Initialization:} $\tilde{\mathbf{x}}^{(0)}:=\mathbf x$. 
		\FORALL{$(j,i) \in \mathcal{S}({\mathbf x}) \times \mathcal{Z}({\mathbf x})$}
    	\STATE  Compute $\alpha^*_{(j,i),\mathbf{x}}$ and $M_{(j,i),\mathbf{x}}:=\tilde{R}_{(j,i),\mathbf{x}}(\alpha^*_{(j,i), \mathbf{x}})$.
		\ENDFOR
		\FOR{$t=1$ {\bfseries to} $r$}
    	\STATE $(j_t,i_t) := \mathop{\arg \max}_{(j,i) \in (\mathcal{S}({\mathbf x}) \backslash \{j_1,\dots,j_{t-1}\}) \times (\mathcal{Z}({\mathbf x}) \backslash \{i_1,\dots,i_{t-1}\})} M_{(j,i),\mathbf{x}}$.
    	\STATE $\tilde{\mathbf{x}}^{(t)}:=\begin{cases}
          \tilde{\mathbf{x}}^{(t-1)} - [\tilde{\mathbf{x}}^{(t-1)}]_{j_t} \mathbf{e}_{j_t} + \alpha_{(j_t,i_t),\tilde{\mathbf{x}}^{(t-1)}}^* \mathbf{e}_{i_t}&, \alpha_{(j_t,i_t),\tilde{\mathbf{x}}^{(t-1)}}^*\ne \pm\infty,\\
          \mathbf{e}_{i_t}&, \alpha_{(j_t,i_t),\tilde{\mathbf{x}}^{(t-1)}}^*=\pm\infty,\end{cases}.$

		\ENDFOR
		\STATE {\bfseries Output:} an support altered solution $\tilde{\mathbf x}^{(r)}$
	\end{algorithmic}
\end{algorithm}
The inputs of greedy support alteration are a vector $\mathbf{x}\in \mathbb{R}^n$ and the number of alterations $r$ to be acted on $\mathbf{x}$. It requires $r \le \min\{\|\mathbf{x}\|_0,n-\|\mathbf{x}\|_0\}$. The output of the algorithm is denoted by $\tilde{\mathbf x}^{(r)}$, the support altered vector of $\mathbf{x}$. Algorithm \ref{alg:greedy} has two sequential for-loops. In the first for-loop, we compute  $M_{(j,i),\mathbf{x}}=\tilde{R}_{(j,i),\mathbf{x}}(\alpha^*_{(j,i),\mathbf{x}})$, that is the objective value achieved by changing $[\mathbf{x}]_j$ to zero and $[\mathbf{x}]_i$ to $\alpha^*_{(j,i),\mathbf{x}}$ for $(j,i) \in \mathcal{S}({\mathbf{x}}) \times \mathcal{Z}({\mathbf{x}})$. In the second for-loop, we ideally look for an altered version of $\mathbf{x}$ by successively identifying $r$ indices to be moved out from $\mathcal{S}({\mathbf{x}})$ and $r$ indices from $\mathcal{Z}({\mathbf{x}})$ to be potentially moved into $\mathcal{S}(\tilde{\mathbf{x}})$. This is done through $r$ iterations. Each iteration identifies a pair of indices and generates an updated version of $\mathbf{x}$. The concrete process is as follows: suppose the identified index pairs are $(j_1,i_1),\dots,(j_{t-1},i_{t-1})$ in the previous $(t-1)$ iterations and $\tilde{\mathbf{x}}^{(t-1)}$ is the altered solution at the $(t-1)$-th iteration. The pair of indices at the $t$-th iteration $(j_t,i_t)$ is selected from $(\mathcal{S}({\mathbf x}) \backslash \{j_1,\dots,j_{t-1}\}) \times (\mathcal{Z}({\mathbf x}) \backslash \{i_1,\dots,i_{t-1}\})$ such that $M_{(j_t,i_t),\mathbf{x}}$ is the largest value among all possible choices. Then, the updated vector $\tilde{\mathbf{x}}^{(t)}$ from $\tilde{\mathbf{x}}^{(t-1)}$ is given as follows:
\begin{equation}\label{eq:twolevels}
\tilde{\mathbf{x}}^{(t)}:=\begin{cases}
          \tilde{\mathbf{x}}^{(t-1)} - [\tilde{\mathbf{x}}^{(t-1)}]_{j_t} \mathbf{e}_{j_t} + \alpha_{(j_t,i_t),\tilde{\mathbf{x}}^{(t-1)}}^* \mathbf{e}_{i_t}&, \alpha_{(j_t,i_t),\tilde{\mathbf{x}}^{(t-1)}}^*\ne \pm\infty,\\
          \mathbf{e}_{i_t}&, \alpha_{(j_t,i_t),\tilde{\mathbf{x}}^{(t-1)}}^*=\pm\infty,\end{cases}.
\end{equation}
Figure \ref{fig:greedy} illustrates this process for $r=2$. 
\begin{figure}[htbp]
    \centering
    \includegraphics[scale=0.7,trim=40 620 80 50]{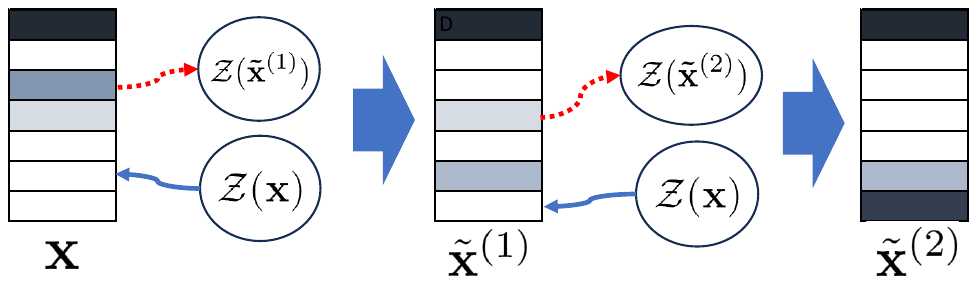}
    \caption{Illustration of greedy support alteration, where support alteration of $r=2$ is done by swapping one pair of out-of-support index and into-support index at a time.}
    \label{fig:greedy}
\end{figure}

The support alteration method given in Algorithm~\ref{alg:greedy} computes the pairs of indices with the largest values of $M_{(j,i),\mathbf x}$ from all possible $(j,i)$, which requires $s(n-s)$ executions of one-dimensional maximization procedures. We can further reduce the computational cost with the following frequently used idea that the smaller the absolute values of nonzero entries with indices in $\mathcal{S}({\mathbf x})$ are, the less important they are. Our second variant of support alteration, termed \textit{partial support alteration} and presented in Algorithm \ref{alg:Partial}, is stemmed from this idea. 
\begin{algorithm}[htbp]
	\caption[Partial]{(\texttt{SA}) Partial support alteration}	
	\begin{algorithmic}
		\label{alg:Partial}
		\STATE {\bfseries Input:} a solution $\mathbf x\in \mathbb{R}^n$ and $r$ index pairs to be swapped with $r \le \min\{\|\mathbf{x}\|_0,n-\|\mathbf{x}\|_0\}$
		\STATE {\bfseries Initialization:} $\tilde{\mathbf{x}}^{(0)}:=\mathbf x$. 
    	\STATE Find $\{j_1,\dots,j_r\}$ such that $0<|[\mathbf x]_{j_1}| \le |[\mathbf x]_{j_2}| \le ...\le|[\mathbf x]_{j_r}|\le |[\mathbf x]_j|$, $\forall j\in \mathcal{S}(\mathbf{x})\setminus \{j_1,\dots,j_r\}$.
		\FOR{$t=1$ {\bfseries to} $r$}
      	\STATE $i_t:=\mathop{\arg \max}_{i \in \mathcal{Z}({\mathbf x}) \backslash \{i_1,\dots,i_{t-1}\}} M_{j_t,i, \tilde{\mathbf{x}}^{(t-1)}}$.
      	\STATE $\tilde{\mathbf{x}}^{(t)}:=\begin{cases}
          \tilde{\mathbf{x}}^{(t-1)} - [\tilde{\mathbf{x}}^{(t-1)}]_{j_t} \mathbf{e}_{j_t} + \alpha_{(j_t,i_t),\tilde{\mathbf{x}}^{(t-1)}}^* \mathbf{e}_{i_t}&, \alpha_{(j_t,i_t),\tilde{\mathbf{x}}^{(t-1)}}^*\ne \pm\infty,\\
          \mathbf{e}_{i_t}&, \alpha_{(j_t,i_t),\tilde{\mathbf{x}}^{(t-1)}}^*=\pm\infty,\end{cases}.$
		\ENDFOR
		\STATE {\bfseries Output:} $\tilde{\mathbf x}^{(r)}:=\texttt{SA}(\mathbf{x}, r)$  the support altered solution of $\mathbf{x}$. 
	\end{algorithmic}
\end{algorithm}
The partial support alteration determines the set of out-of-support indices to be those $r$ entries of $\mathbf{x}$ with the smallest absolute values. It only considers a part of choices when selecting index pairs $(j,i)$, while the greedy support alteration considers all choices. In details, the partial support alteration first finds $r$ indices corresponding to the smallest $r$ nonzero entries in absolute value of $\mathbf{x}$, denoted by $j_1,\dots,j_r$, where $|[\mathbf x]_{j_1}| \le |[\mathbf x]_{j_2}| \le \ldots \le|[\mathbf x]_{j_r}|$. In the for-loop, it successively removes those indices from $\mathcal{S}({\mathbf x})$, and add potentially equal number of indices from $\mathcal{Z}({\mathbf x})$ into $\mathcal{S}({\mathbf x})$. This is done successively as follows. At the $t$-th iteration, the partial support alteration fixes $j_t$ as the index of the $t$-th smallest  nonzero entries of $\mathbf x$ in absolute value, and enumerates all into-support indices from the set $\mathcal{Z}({\mathbf x}) \backslash \{i_1,\dots,i_{t-1}\}$ to find the available into-support index corresponding to the largest objective value. The $t$-th altered $\tilde{\mathbf{x}}^{(t)}$ is updated using the formula~\eqref{eq:twolevels}. This process is visually illustrated in Figure \ref{fig:partial}.


\begin{figure}[htbp]
    \centering
    \includegraphics[scale=0.7,trim=40 620 80 50]{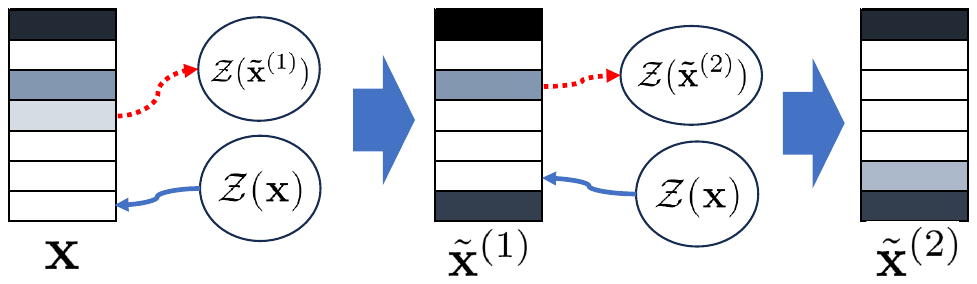}
    \caption{Illustration of partial support alteration. The darkness of the entries represent the absolute values of the entries in a vector. The indices corresponding to the smallest absolute values of the entries, represented by the lightest colors, are selected to be the out-of-support indices.}
    \label{fig:partial}
\end{figure}

As partial support alteration only considers a subset of the possible choices for index pairs, the number of required one-dimensional maximization procedures reduces. To count the number of one-dimensional optimizations needed, consider the $t$-th iteration of the for-loop in the partial support alteration, where it finds the into-support index in $\mathcal{Z}({\mathbf x}) \backslash \{i_1,\dots,i_{t-1}\}$ corresponding to the largest objective value. This requires solving a one-dimensional optimization problem \eqref{bestalpha} $|\mathcal{Z}({\mathbf x}) \backslash {i_1,\dots,i_{t-1}}|=n-s-(t-1)$ times. Therefore, the total number of one-dimensional optimizations to be solved for $t$ ranging from $1$ to $r$ is
$$
\sum_{t=1}^r (n-s-(t-1)) = r\left(n-s-\frac{r-1}{2}\right),
$$
which is less than $s(n-s)$ required by greedy support alteration.

\section{The Successive Two-Stage Framework with Self-Adaptive Number of Swapped Pairs}
\label{sec pair number}
The number of swapped pairs, denoted by $r$ in Algorithm \ref{alg:greedy} and \ref{alg:Partial}, is a key parameter in the support alteration and greatly affects the outcome of the successive two-stage framework for sGEP. In this section, we first propose a strategy to self-adaptively determine the number $r$ of swapped pairs for the support alteration in each iteration. Then, we present the whole algorithmic framework with such self-adaptive strategy and analyze its convergence property.
\subsection{A Self-Adaptive Strategy for Determining the Number of Swapped Pairs}  
This subsection is devoted to the self-adaptive strategy for determining the number of swapped pairs, i.e., $r$ in Algorithm \ref{alg:greedy} and Algorithm \ref{alg:Partial}. The spirit of this strategy comes from the backtracking strategy, a classical idea to adaptively adjust the step size of the underlying algorithm. In the following, we will first explain that the number of swapped pairs can be viewed as the step size of the support alteration when the $\ell_0$ norm is used to measure the distance between the output and input of the support alteration. To this end, we first show that the $\ell_0$ norm is a metric on $\mathbb{R}^n$ in the next proposition. Define $\tau: \mathbb{R}^n \times \mathbb{R}^n \rightarrow \mathbb{N}$ at $(\mathbf{x},\mathbf{y})\in \mathbb{R}^n\times \mathbb{R}^n$ as
$$
\tau(\mathbf{x},\mathbf{y}):=\|\mathbf{x}-\mathbf{y}\|_0.
$$ 
Immediately, we have the following result.
\begin{Proposition}
\label{prop metric}
$(\mathbb{R}^n, \tau)$ is a metric space, where $\tau$ is a metric on $\mathbb{R}^n$. 
\end{Proposition}
\begin{proof} \ \ The proof has three parts. 

(i) For all $\mathbf{x}$ and $\mathbf{y}$ in $\mathbb{R}^n$, $\tau(\mathbf{x},\mathbf{y}) \ge 0$ and  $\tau(\mathbf{x},\mathbf{y}) = 0$ if and only if $\mathbf{x}=\mathbf{y}$. It is true from the property of the $\ell_0$ norm. 

(ii)  For all $\mathbf{x}$ and $\mathbf{y}$ in $\mathbb{R}^n$, $\tau(\mathbf{x},\mathbf{y}) = \tau(\mathbf{y},\mathbf{x})$. This statement is a direct consequence of the definition of $\tau$. 

(iii) For all $\mathbf{x}$, $\mathbf{y}$, and $\mathbf{z}$ in $\mathbb{R}^n$, $\tau(\mathbf{x},\mathbf{y}) \le \tau(\mathbf{x},\mathbf{z}) +\tau(\mathbf{z},\mathbf{y})$. Actually, this triangle inequality comes from the fact that $\|\mathbf{x}-\mathbf{y}\|_0 \le \|\mathbf{x}-\mathbf{z}\|_0 +\|\mathbf{z}-\mathbf{y}\|_0$.

Therefore, $\tau$ is a metric on $\mathbb{R}^n$, which completes the proof of the result. 
\end{proof}

The metric $\tau$ can help us to understand and determine the parameter $r$. Consider an iterative algorithm in optimization, 
let $\mathbf{x}^{(t)}$ be the iterate in the $t$-th iteration, the next iterate is 
$\mathbf{x}^{(t)}:=\mathbf{x}^{(t-1)}+\alpha_t \mathbf{d}_t$, where  $\mathbf{d}_t$ is a unit direction vector measured by the $\ell_2$ norm. The quantity $\alpha_t=\|\mathbf{x}^{(t)}-\mathbf{x}^{(t-1)}\|_2$ is called the step size of this algorithm. By replacing the $\ell_2$ norm with  the $\ell_0$ norm, $\tau(\mathbf{x},\hat{\mathbf{x}})=\|\mathbf{x}-\hat{\mathbf{x}}\|_0$ can be viewed as the step size of a support alteration procedure, where $\mathbf{x}$ and $\hat{\mathbf{x}}$ are the input and output of the support alteration, respectively. Suppose $\|\mathbf{x}\|_0=\|\hat{\mathbf{x}}\|_0$, then we immediately have $\tau(\mathbf{x},\hat{\mathbf{x}})= 2r$. Hence we can also view $r$ as the step size of the support alteration. 


We now propose an adaptive approach to find the suitable step size for the support alteration. The pseudo-code of this approach is displayed in Algorithm \ref{alg:stepsize}. In this code, $\texttt{GBA}(\cdot)$ denotes a gradient-based algorithm, such as Algorithm \ref{alg:PGSA}, and $\texttt{SA}(\mathbf{x},r)$ is the resulting vector of swapping $r$ index pairs of $\mathbf{x}$ by a support alteration, such as Algorithm \ref{alg:Partial}. For simplicity, we omit the arguments $\mathbf{A}$ and $\mathbf{B}$ of the function $R(\cdot)$ since they are known in the context. Algorithm \ref{alg:stepsize} has two inputs, a vector $\mathbf{x}$ to be altered and a number $\hat{r}$, which is the maximum to start from for a support alteration. The algorithm determines a step size $r$ for the support alteration. This $r$ is the maximum number not exceeding $\hat{r}$ ensuring that the support alteration of the input vector $\mathbf{x}$ results in an improvement in the sense that $R(\texttt{GBA}(\texttt{SA}(\mathbf{x},r))) > R(\mathbf{x})$.


\begin{algorithm}[htb]
	\caption{(\texttt{SS}) A self-adaptive strategy to determine the Step Size for support alteration}	
	\begin{algorithmic}
		\label{alg:stepsize}
		\STATE {\bfseries Input:} $\mathbf{x}$ a vector to be altered; $\hat{r}$ the maximum number to start from.
        \STATE {\bfseries Initialization:} $r:=\hat{r}$.
        \WHILE{$r\ne 0$ and $R(\texttt{GBA}(\texttt{SA}(\mathbf{x},r))) \le R(\mathbf{x})$}
        \STATE $r:=r-1$.
        \ENDWHILE
        \STATE {\bfseries Output:} $r=\texttt{SS}(\mathbf{x},\hat{r})$ the determined step size for support alteration.
	\end{algorithmic}
\end{algorithm}
\subsection{Convergent Successive Two-Stage Framework with Self-Adaptive Number of Swapped Pairs}
In this subsection, we present our successive two-stage algorithm with the self-adaptive step size in Algorithm \ref{alg:proposed} and analyze its convergence property. 

\begin{algorithm}[htb]
	\caption{(\texttt{SA}\_\texttt{GBA}) A successive two-stage algorithm with self-adaptive number of swapped pairs}	
	\begin{algorithmic}\label{alg:proposed}
		\STATE {\bfseries Input:} $\mathbf{A} \in \mathbb{S}_+^n$, $\mathbf{B} \in \mathbb{S}_{++}^n$, and sparsity level $s$, a gradient-based algorithm $\texttt{GBA}(\cdot)$.
		\STATE {\bfseries Initialization:} $t:=0$. Find and denote by $\hat{\mathbf{x}}^{(0)}$, an initial point for $\texttt{GBA}(\cdot)$.
		\REPEAT
            \STATE $t:=t+1$.
            \STATE Stage-1: $\mathbf{x}^{(t)}:=\texttt{GBA}({\hat{\mathbf{x}}}^{(t-1)})$. 
            \STATE Stage-2: 
            $\hat{r}_t:=\begin{cases}
                \min\{\|\mathbf{x}^{(1)}\|_0,n-\|\mathbf{x}^{(1)}\|_0\}, & t=1,\\
                \min\{r_{t-1}-1,\|\mathbf{x}^{(t)}\|_0,n-\|\mathbf{x}^{(t)}\|_0\}, & t \ge 2.
            \end{cases}$\\
            $r_t:=\texttt{SS}(\mathbf{x}^{(t)},\hat{r}_t)$.\\
            ${\hat{\mathbf{x}}}^{(t)}:=\texttt{SA}(\mathbf{x}^{(t)},r_t)$. 
		\UNTIL {$r_t=0$}
		\STATE {\bfseries Output:} $\mathbf{x}^{(t)}$
	\end{algorithmic}
\end{algorithm}

We give some remarks on this algorithm. First, recall that the operator $\texttt{SA}(\cdot)$ denotes a call to the partial support alteration Algorithm~\ref{alg:Partial} while the operator $\texttt{SS}(\cdot)$  denotes a call to Algorithm~\ref{alg:stepsize} for finding the step size of the support alteration. Second, we will rename Algorithm \ref{alg:proposed} to reflect the chosen \texttt{GBA} in the section of numerical experiments. For example, for the two given \texttt{GBA}s, TPM and PGSA\_ML, in Section~\ref{sec framework}, Algorithm \ref{alg:proposed} is renamed as SA\_TPM (or SA\_PGSA\_ML) if \texttt{GBA} is TPM (or PGSA\_ML). 

The idea behind Algorithm \ref{alg:proposed} is motivated from both diminishing step size strategy and backtracking strategy, two classical step size strategies measured in the $\ell_2$ norm. The diminishing step size strategy gradually reduces the step size as the algorithm progresses. Utilizing this idea in our case, we first initialize the step size, $r_1$, as the maximum feasible value, $r_1 := \min\{\|\mathbf{x}^{(1)}\|_0, n-\|\mathbf{x}^{(1)}\|_0\}$. In the $t$-th iteration ($t\ge 2$) of Algorithm \ref{alg:proposed}, the searching step size starts from $\min\{r_{t-1}-1,\|\mathbf{x}^{(t)}\|_0,n-\|\mathbf{x}^{(t)}\|_0\}$ making sure $r_t\le r_{t-1}-1$. Hence, the step size sequence obtained in Algorithm \ref{alg:proposed}, denoted by $\{r_t\}$, strictly decreases to 0. 

The backtracking strategy, typically used to find a suitable step size ensuring strictly increasing objective values, is adopted in our algorithm with modifications. In the classical backtracking strategy for iterative algorithms, the criterion to decide if backtracking is needed is to simply compare the objective values. However, in Algorithm \ref{alg:stepsize} the criterion to decide whether backtracking is needed is to check if $R(\texttt{GBA}(\texttt{SA}(\mathbf{x},r))) > R(\mathbf{x})$, which means if the objective value achieved by a gradient-based algorithm using the altered vector as the initial point is higher than the objective value obtained by the unaltered one. 
There are two primary motivations for this modification. Firstly, although the altered solution, represented by $\texttt{SA}(\mathbf{x},r)$, may exhibit a favorable support, it might not necessarily reflect the optimal entry values within that support. Given that the input vector $\mathbf{x}$ is a stationary point of \eqref{sGEP}, $\mathbf{x}$ usually contains optimal or near-optimal values for its support. Hence, our objective is to enhance the values within the altered support for a more balanced comparison. Employing a gradient-based algorithm, symbolized as $\texttt{GBA}(\cdot)$, offers a cost-efficient method to refine those values.
Secondly, the primary objective of support alteration is to identify an improved starting point for a gradient-based algorithm, rather than to directly yield an enhanced solution. Since the chosen gradient-based algorithm is monotonic, the criterion $R(\texttt{GBA}(\texttt{SA}(\mathbf{x},r))) > R(\mathbf{x})$ is more likely to be satisfied than $R(\texttt{SA}(\mathbf{x},r)) > R(\mathbf{x})$. Note that the chosen step size is the largest number that satisfies the criterion. This adjustment favors a larger support alteration step size than conventional backtracking. As a result, the distance between $\texttt{SA}(\mathbf{x},r)$ and $\mathbf{x}$ increases, enhancing the likelihood of escaping local minima.

The properties of Algorithm \ref{alg:proposed} are listed in the following theorem.
\begin{theorem}
Let $\{\mathbf{x}^{(t)}\}$ be the sequence generated by Algorithm \ref{alg:proposed}. The following statements about Algorithm \ref{alg:proposed} hold.
\begin{enumerate}[label=(\roman*)]
    \item The sequence of objective values $\{R(\mathbf{x}^{(t)})\}$ strictly increases until halting.
    \item The algorithm stops in at most $s$ iterations.
\end{enumerate}
\end{theorem}
\begin{proof}
Suppose $t$ is an iteration number.

(i) 
Recall the description of Algorithm \ref{alg:proposed} and the stopping criterion of Algorithm \ref{alg:stepsize}.
If $r_t>0$, then 
\[R(\mathbf{x}^{(t+1)})=R(\texttt{GBA}(\hat{\mathbf{x}}^{(t)}))=R(\texttt{GBA}(\texttt{SA}(\mathbf{x}^{(t)},r_t)))>R(\mathbf{x}^{(t)}).\] 
If $r_t=0$, from the description of Algorithm \ref{alg:proposed}, Algorithm \ref{alg:proposed} halts.

(ii)  From the descriptions of Algorithm \ref{alg:stepsize} and Algorithm \ref{alg:proposed}, we have
    \[0\le r_t\le r_{t-1}-1\le \cdots \le r_1-t\le \|\mathbf{x}^{(1)}\|_0-t\le s-t.\]
    Hence, an iteration number must satisfy  $t\le s$.
\end{proof}

To close this section, we underscore the distinct advantages of our approach over coordinate-wise algorithms by drawing comparisons across four key aspects:

\textit{Swapping Multiple Entry Pairs:} Unlike coordinate-wise algorithms, which only swap one pair at a time \cite{beck2013sparsity,beck2016sparse}, our method is capable of swapping multiple entry pairs in a single step. As we elucidate in Section \ref{sec pair number}, this strategy can increase the $\ell_0$ distance between the input and the output. This, in turn, augments the likelihood of the solution escaping a local optimum.

\textit{Optimal Value Assignment:} Instead of merely assigning the swapped-out value to the swapped-in entry as done in coordinate-wise algorithms, our technique assigns an approximate optimal value. This proves to be a more precise approximation, especially when multiple entry pairs are swapped.

\textit{Gradient-based Optimization:} Our approach adopts a gradient-based method in its initial stage. This method not only improves the objective value but also increases the sparsity level of the iterates to $s$. This feature is absent in coordinate-wise algorithms.

\textit{Scope:} While our method is tailored to solve the sGEP, coordinate-wise algorithms are limited to sPCA.

\section{Experiments} \label{sec experiment}

Various algorithms have been provided for sGEP and some of them have been reviewed in the introduction section. In this section, these different methods are compared with our proposed algorithm. For sPCA, our algorithm is denoted by SA\_TPM representing employ TPM as our $\texttt{GBA}$ in the first stage. The  compared algorithms are TPM \cite{yuan2013truncated}, PathSPCA \cite{d2008optimal} (denoted by Path), GCW, and PCW \cite{beck2016sparse}. For general sGEP cases, our algorithms are denoted by SA\_PGSA\_ML, SA\_rifle respectively. The compared algorithms are rifle \cite{tan2018sparse}, PGSA\_ML \cite{zhang2022first}, IFTRR \cite{cai2020inverse}, IRQM with logarithm surrogate \cite{song2015sparse} (denoted by IRQM\_log), and DEC \cite{yuan2019decomposition}. For a fair comparison, we also aim to maximize the performance of some of the compared algorithms that solve the penalized formulation of sGEP:
\begin{equation}
	\max_{\mathbf{x} \in \mathbb{R}^n}\{\mathbf{x}^\top \mathbf{A} \mathbf{x} - \rho \|\mathbf{x}\|_0: \mathbf{x}^\top \mathbf{B} \mathbf{x} = 1\}, \label{penalized}
\end{equation} such as IRQM \cite{song2015sparse}. Those algorithms require an appropriate hyperparameter $\rho$ to obtain the desirable sparsity $s$. The parameters $\rho$ and $s$ are unknown. To maximize the performance of the compared algorithms solving problem \eqref{penalized}, these algorithms are applied many times to search the suitable $\rho^*$ for the desired $s^*$, and finally use those correspondence in experiments. 

All experiments were performed in MATLAB on a PC with an i9-13900K CPU and 32GB RAM. The experiments consist of several special cases and applications of sGEP, including sPCA, sparse recovery analysis, sFDA for classifications, sCCA, and sSIR.


\subsection{Sparse Principal Component Analysis}
The sPCA aims to obtain the sparse direction that maximizes the variance of the data after being projected along this direction.  Let $\mathbf{x}_{sPCA}$ be the sparse eigenvector extracted by sPCA algorithms and let $\lambda_{\mathbf{A}}^{\max}$ be the largest eigenvalue of the sample covariance matrix $\mathbf{A}$. We use the proportion of explained variance as a measure of comparison, which is defined as
\[\left(\frac{\mathbf{x}_{sPCA}^\top \mathbf{A} \mathbf{x}_{sPCA}}{\mathbf{x}_{sPCA}^\top\mathbf{x}_{sPCA}}\right) / \lambda_{\mathbf{A}}^{\max}. \] The proportion of explained variance is a commonly used measure in sPCA. A model with high explained variance will have good predictive power. And the proportion of explained variance is directly proportional to the corresponding Rayleigh quotient. The competing algorithms and our proposed algorithm are tested on two different types of data sets, one is the Pitprops real data set \cite{jeffers1967two} and the other is the Gaussian random data set. The Pitprops data is a correlation matrix that was calculated from 180 observations. There are 13 explanatory variables. Results up to sparsity level of 12 are shown in Figure \ref{fig:pitprops}. All the algorithms perform similarly when the sparsity level is within the ranges $[1, 3]$ and $[7, 12]$. Our algorithm, SA\_TPM, maintains the highest objective value among all competing algorithms for all sparsity. It provides a solid improvement from TPM for sparsity in the range $[4,6]$, and also outperforms its counterparts, GCW and PCW.

\begin{figure}[htb]
	\centering
	\includegraphics[scale=0.55,trim=0 0 0 0]{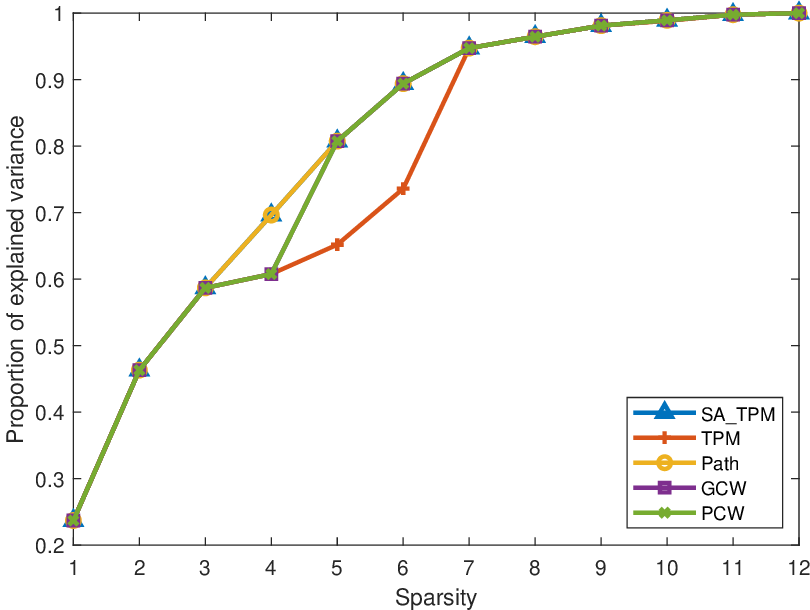}
	\caption{Proportion of explained variance vs sparsity for ``Pitprops" data set.}
    \label{fig:pitprops}
\end{figure}

The random data set is generated as follows: First, we generate a random data matrix $\mathbf{D} \in \mathbb{R}^{m \times n}$, $[\mathbf{D}]_{i,j} \sim \mathcal{N}(0,1)$. Then, let $\mathbf{A}$ be the covariance matrix corresponding to $\mathbf{D}$. The number of samples is set to be $m=300$. Results up to sparsity level of 400 are shown in Figure \ref{fig:pev_gaussian} (left) for the number of variables $n=5000$ and in Figure \ref{fig:pev_gaussian} (right) for the number of variables $n=20000$. The result shown in Figure \ref{fig:pev_gaussian} is the average result from 100 different data sets. 
\begin{figure}[htb]
	\centering
    \includegraphics[scale=0.55,trim=10 30 0 0]{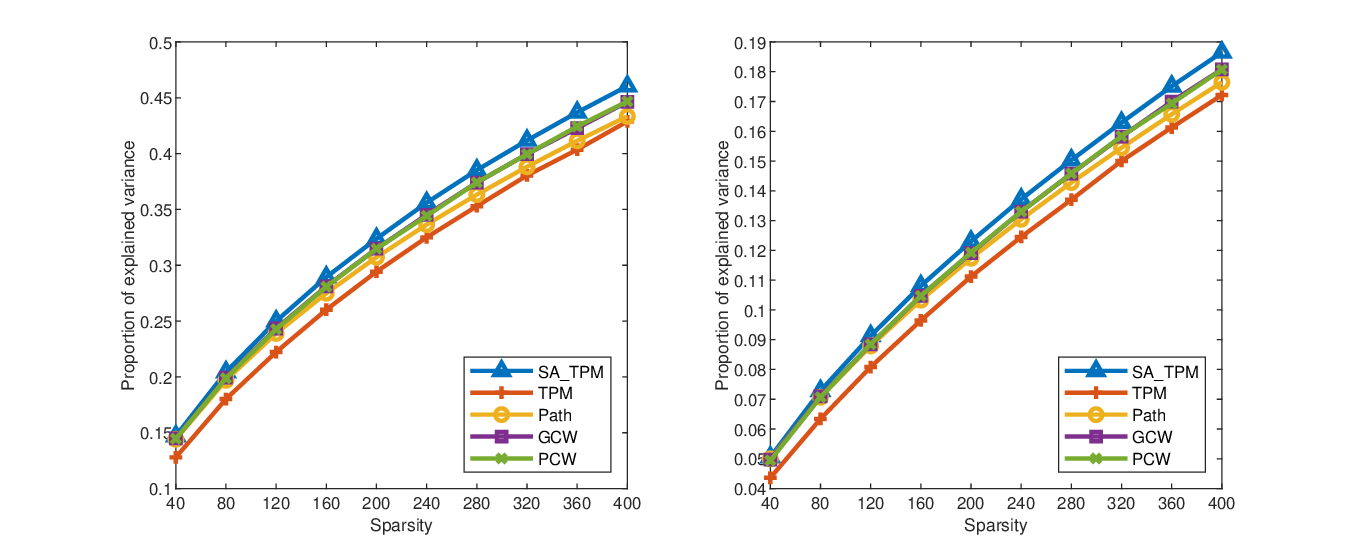}
	\caption{Proportion of explained variance vs sparsity on Gaussian data set with $n=5000$ (left) and $n=20000$ (right).  100 simulations are used to produce all results.}
	\label{fig:pev_gaussian}
\end{figure}
In both cases, where the number of variables varies, our method significantly improves the objective value compared to that produced by TPM, and also excels compared to competing algorithms.

For the Gaussian random data sets used in Figure~\ref{fig:pev_gaussian},  the computational time comparing the various algorithms for the different levels of sparsity is displayed in Figure \ref{fig:time} (left) for $n=5000$ and in Figure \ref{fig:time} (right) for $n=20000$.   
\begin{figure}[htb]
	\centering
    \includegraphics[scale=0.47,trim=10 30 0 0]{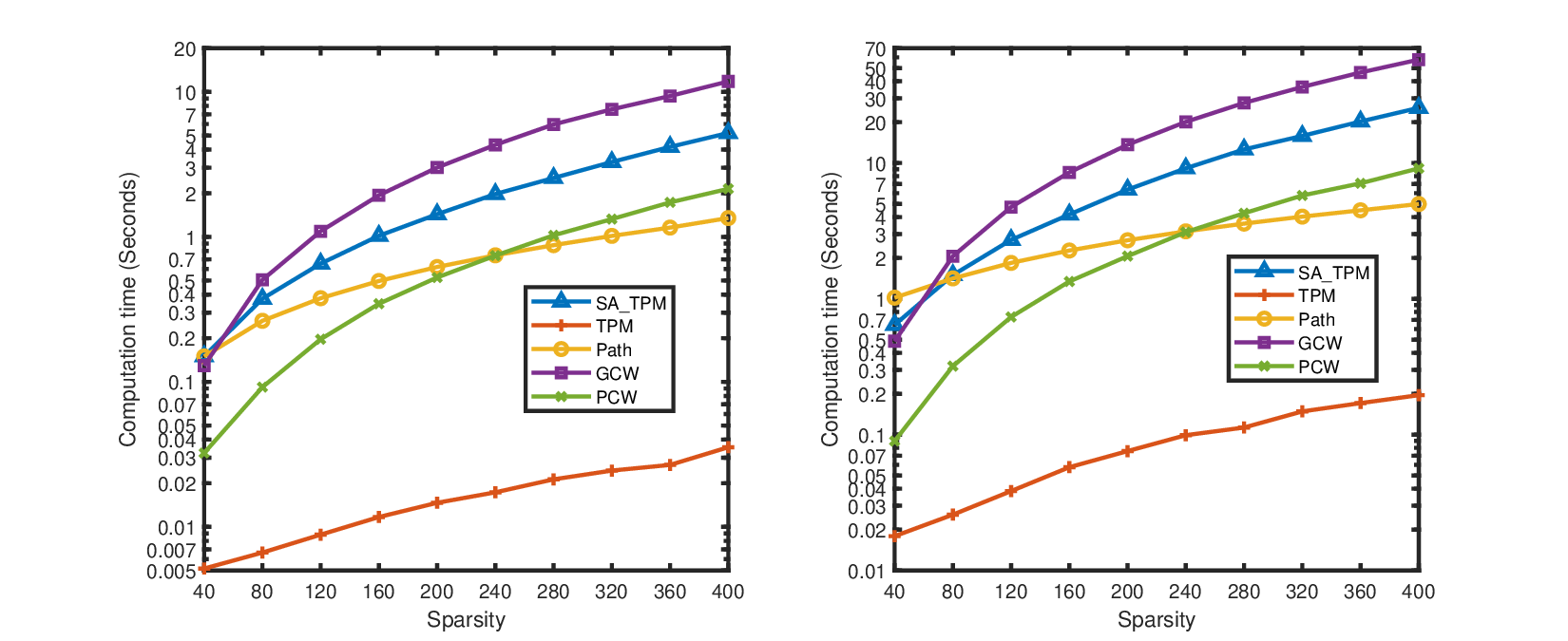}
	\caption{Computational Time vs sparsity on Gaussian data set with $n=5000$ (left) and $n=20000$ (right).}
	\label{fig:time}
\end{figure}
The computation time of SA\_TPM ranges from about 0.1s ($s=40,n=5000$) to about 25s ($s=400,n=20000$), which is typically above the one of PCW, but less than that of GCW.
Figure~\ref{fig:time_dim} displays the computational time of our proposed algorithm applied to the Gaussian data set described above, but with varying dimensions. 
\begin{figure}[ht]
	\centering
    \includegraphics[scale=0.6,trim=0 0 0 0]{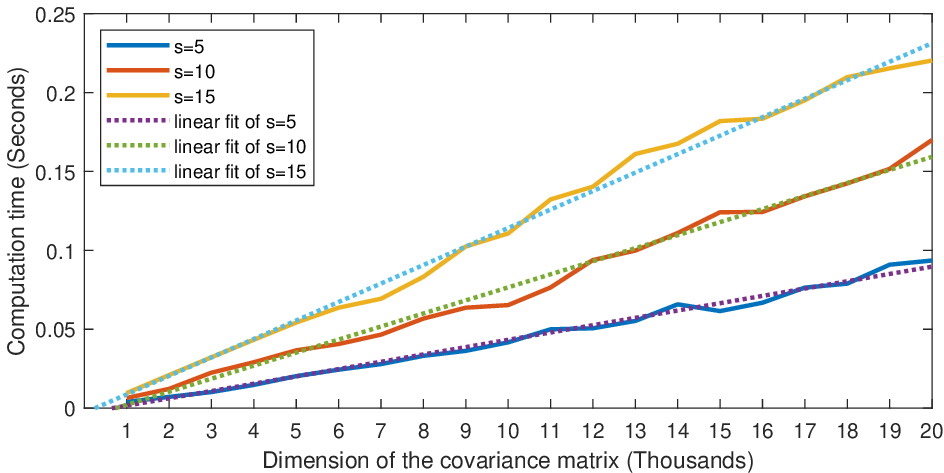}
	\caption{Computation Time vs dimension on Gaussian data set. 100 simulations are used to produce all results.}
	\label{fig:time_dim}
\end{figure}
Figure \ref{fig:time_dim} shows that our method is able to converge within 0.1s for high-dimensional ($n=20000$) cases with low sparsity cases ($s=5$). Dash lines in Figure \ref{fig:time_dim} are linear fittings of cases with fixed sparsities but varying dimensions. Linear regressions show that the computation time of our method increases linearly with respect to dimensions. In summary, the computation time of our method increases rapidly with respect to sparsity $s$, but only increases linearly with respect to dimension $n$. This aligns with the conclusion drawn in Section \ref{sec:existence}, that the complexity of sGEP increases more quickly with respect to $s$ than it does with respect to $n$. 

\subsection{Sparse Recovery Analysis}
We consider here the sparse spike model with Gaussian noise. The sparse spike model in \cite{shen2008sparse} aims to generate random data with the covariance matrix consisting of a sparse dominant eigenvector. Our model is slightly different in that the observed data is additionally tainted by Gaussian noise. The covariance matrix is synthesized through an eigendecomposition $\bm\Sigma=\mathbf{V} \bm\Lambda \mathbf{V}^{\top}$, where the first column of $\mathbf{V}\in \mathbb{R}^{n \times n}$ is a specified sparse unit vector. The ground truth data matrix $\mathbf{X} \in \mathbb{R}^{m \times n}$ is generated by drawing $m$ samples according to the zero-mean normal distribution with covariance matrix $\bm\Sigma$. Then the observed data matrix $\tilde{\mathbf{X}}$ equals to $\mathbf{X}$ plus i.i.d Gaussian noise with standard deviation $\sigma$, that is $\tilde{\mathbf{X}}=\mathbf{X}+\mathbf{N}$, where $[\mathbf{N}]_{i,j} \sim \mathcal N(0,\sigma^2)$ for all $i \in \mathbb{N}_m^+$, $j\in \mathbb{N}_n^+$. The dominant eigenvector of the ground truth covariance matrix is a random sparse eigenvector:
\begin{equation}
	[\mathbf{v}_1]_i=\begin{cases}\frac{1}{\sqrt{s}}, &i=r_1,...,r_{s}\\0, &\text{otherwise}\end{cases},
\end{equation}
where $r_1,...,r_{s} \in \mathbb{N}_n^+$ are chosen arbitrarily. The eigenvectors $ \mathbf{v}_j $ for $ j \ge 2 $ are chosen as random unit vectors, but are orthogonal to the previous eigenvectors, ensuring that $ \mathbf{V} $ remains an orthogonal matrix. The eigenvalues are fixed as the following values:
\begin{equation}
	\begin{cases}\lambda_1=15,\\ \lambda_j=1, &j=2,\ldots,n.\end{cases}
\end{equation}
Consider the setup with $n=500, \ m=50, \ s=10$. We generate 1000 different matrices and compare the average support recovery rate of each algorithm. Let $\mathbf{u}_1 \in \mathbb{R}^{n}$ be the output solutions of each algorithms, and the support recovery rate in each experiment is computed by $|\mathcal{S}({\mathbf{u}_1}) - \mathcal{S}({\mathbf{v}_1})|/s$. Figure \ref{fig:recovery} illustrates the results.
\begin{figure}[htb]
	\centering
	\includegraphics[scale=0.55,trim=0 0 0 0]{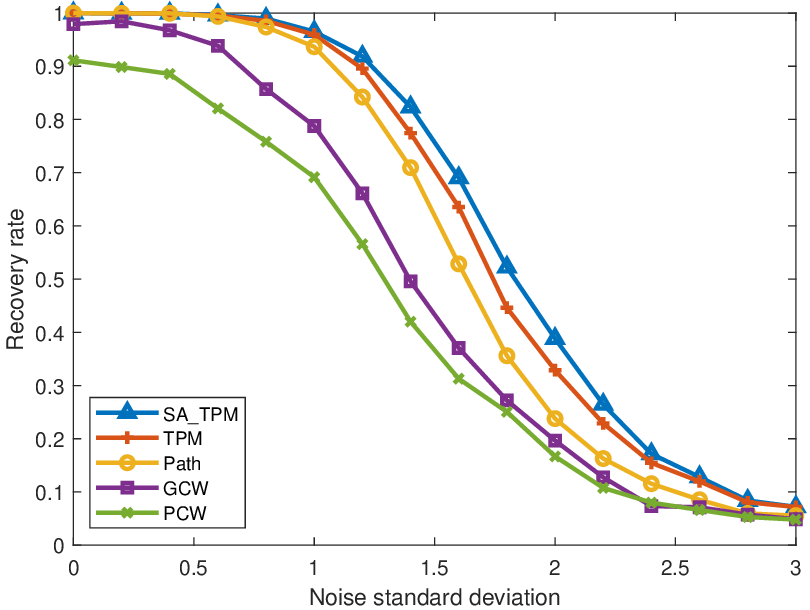}
	\caption{Recovery rate vs noise level in sparse spiked model}
    \label{fig:recovery}
\end{figure}

Our algorithm, SA\_TPM, consistently excels among all competing algorithms. SA\_TPM significantly outperforms GCW and PCW under all noise levels. This is probably due to the fact that the iterations of our algorithm are all enhanced by TPM, considering TPM performs well in sparse recovery setting.
Note that the goal of maximizing the Rayleigh quotient and that of recovering the ground truth sparse eigenvector are generally different. A method that excels in maximizing the Rayleigh quotient value may perform poorly in sparse recovery. However, our method demonstrates excellent recovery rates as well as obtaining high Rayleigh quotient value.

\subsection{Sparse Fisher Discriminant Analysis}
We consider two-class classification problems for both a real data set and a simulated data set. Each data set is further divided into disjoint training and test sets. Each of the algorithm being compared performs an sFDA and use the output to train a classifier with varying sparsity on the training sets. Then we calculate and compare the misclassification rate on the test set achieved by each algorithm. 

Our simulated data set, which is similar to that in \cite{tan2018sparse}, is specified as follows: The mean vector of one class is $\bm \mu_1 =\mathbf{0}$ and the mean vector of the other class is $\bm \mu_2$ with $[\bm\mu_2]_j =0.5$ for $j={2,4,...,40}$ and $[\bm\mu_2]_j =0$ otherwise. Both classes follow the same covariance matrix $\bm\Sigma$. The $\bm\Sigma$ is a block diagonal covariance matrix with five blocks, each of dimension $n/5 \times n/5$. The $(j,j')$-th element of each block takes value $0.8^{|j-j'|}$. This covariance structure is intended to mimic the covariance structure of gene expression data. 
The data are simulated as $\mathbf{x}_k \sim \mathcal{N}(\bm\mu_k,\bm\Sigma), \ k=1,2$. In each simulation, there are 500 training samples and 500 test samples for both classes respectively. We consider the case $n=1000$ and the sparsity levels of each trial are $6,10,\ldots,50$. The results on the simulated data set are averaged over 100 independent trials and are presented below in Figure \ref{fig:rate_simulated}. Our algorithms include SA\_PGSA\_ML and SA\_rifle. They both gain significant improvement from PGSA\_ML and rifle respectively. Additionally, SA\_PGSA\_ML consistently outperforms among all competing algorithms across all sparsity levels.
\begin{figure}[htb]
	\centering
	\includegraphics[scale=0.55,trim=0 0 0 0]{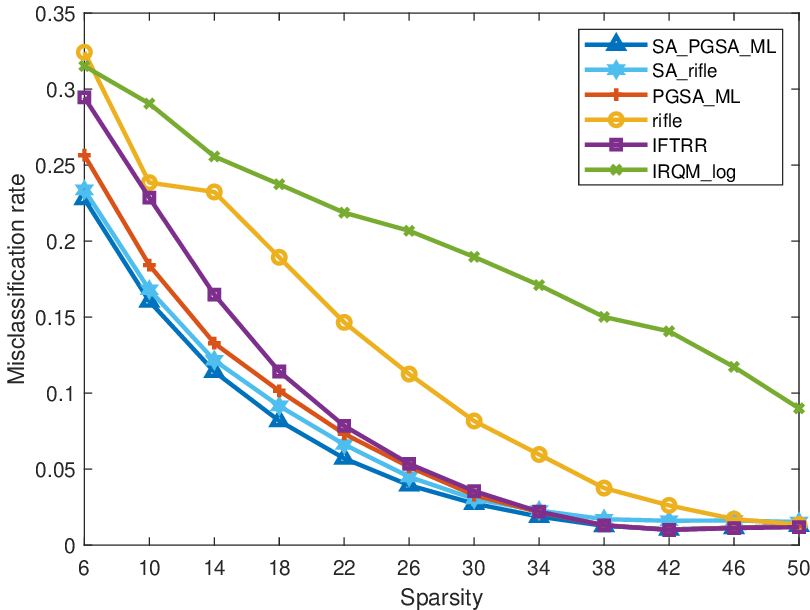}
	\caption{Misclassification rate on test set vs sparsity on simulated data}
    \label{fig:rate_simulated}
\end{figure}

The next sFDA experiment uses a real leukemia data set, which consists of 7129 gene expression measurements from two classes--25 acute myeloid leukemia (AML) patients and 47 acute lymphoblastic leukemia (ALL) patients. Each of the competing algorithms solves an sFDA with a sparsity level of $s=5$ to train classifiers. In each simulation, the data is randomly partitioned, where the training set contains 80\% of the samples, and the remaining 20\% constitutes the test set. This process is independently repeated 50 times using different random splits of training and testing sets. Table \ref{tab:misclassification rate} presents the average results of these 50 independent runs on the leukemia dataset using different training sets and test sets. 
\begin{table}[htb]
	\caption{Average misclassification rate on 50 test sets}
	\label{tab:misclassification rate}
	\resizebox{\textwidth}{!}{%
		\begin{tabular}{|l|l|l|l|l|l|l|l|l|}
			\hline
                Methods                & rifle  & SA\_rifle & PGSA\_ML & SA\_PGSA\_ML & IFTRR & DEC    \\ \hline
			Misclassification Rate & 10.9\% & 6.2\%     & 10.2\%   & 8.0\%  & 11.9\% & 6.9\% \\ \hline
		\end{tabular}%
	}
\end{table}
The table shows the average misclassification rates of various algorithms, allowing for a clear comparison of their performance. From this table, SA\_PGSA\_ML and SA\_rifle gain significant improvements from PGSA\_ML and rifle,  respectively. SA\_rifle gives the lowest misclassification rate among all competing algorithms.

\subsection{Sparse Canonical Correlation Analysis}  
Following the sCCA experiment performed in \cite{tan2018sparse}, we aim to generate a low-rank cross covariance matrix $\bm\Sigma_{XY}=\lambda_1 \bm\Sigma_X \mathbf{v}_X \mathbf{v}_Y^{\top} \bm\Sigma_Y$, where $0<\lambda_1<1$ is the largest generalized eigenvalue of the generalized eigenvalue problem being generated and $\mathbf{v}_X$ and $\mathbf{v}_Y$ are the leading pair of canonical directions of the sCCA. The data consists of two $m \times n/2$ matrices $\mathbf{X}$ and $\mathbf{Y}$. Assume that each row of the two matrices is generated according to $([\mathbf{X}]_{\cdot,:}, [\mathbf{Y}]_{\cdot,:})^\top \sim \mathcal{N}(\mathbf{0},\bm\Sigma)$, where $\bm\Sigma$ is the same as the one in sFDA experiment described above. The sample covariance matrices are then constructed based on the data matrix. In our setting, we set $\lambda_1=0.9$ and let $\mathbf{v}_X, \mathbf{v}_Y$ be sparse vectors, each with sparsity 8 and a random support. Then $\mathbf{v}_X, \mathbf{v}_Y$ are normalized as $\mathbf{v}_X^{\top} \bm\Sigma_X \mathbf{v}_X = \mathbf{v}_Y^{\top} \bm\Sigma_Y \mathbf{v}_Y =1$. And $\bm\Sigma_X,\bm\Sigma_Y$ are both chosen to be the same $n/2 \times n/2$ block diagonal covariance matrix with five blocks, each of dimension $n/10 \times n/10$. The $(j,j')$-th element of each block takes value $0.8^{|j-j'|}$. The measure used here is the covariance between the canonical variables, with a higher value indicating better performance. The sparsity of the experiment starts from $s=16$ since the sparsity of the ground truth leading eigenvector $\mathbf{x}=(\mathbf{v}_X,\mathbf{v}_Y)^\top$ is 16. Figure \ref{fig:sCCA} summarizes the results for the case where $m=200$, $n=1000$, with these results averaged over 100 distinct data sets.
\begin{figure}[h]
	\centering
	\includegraphics[scale=0.43,trim=90 0 90 0]{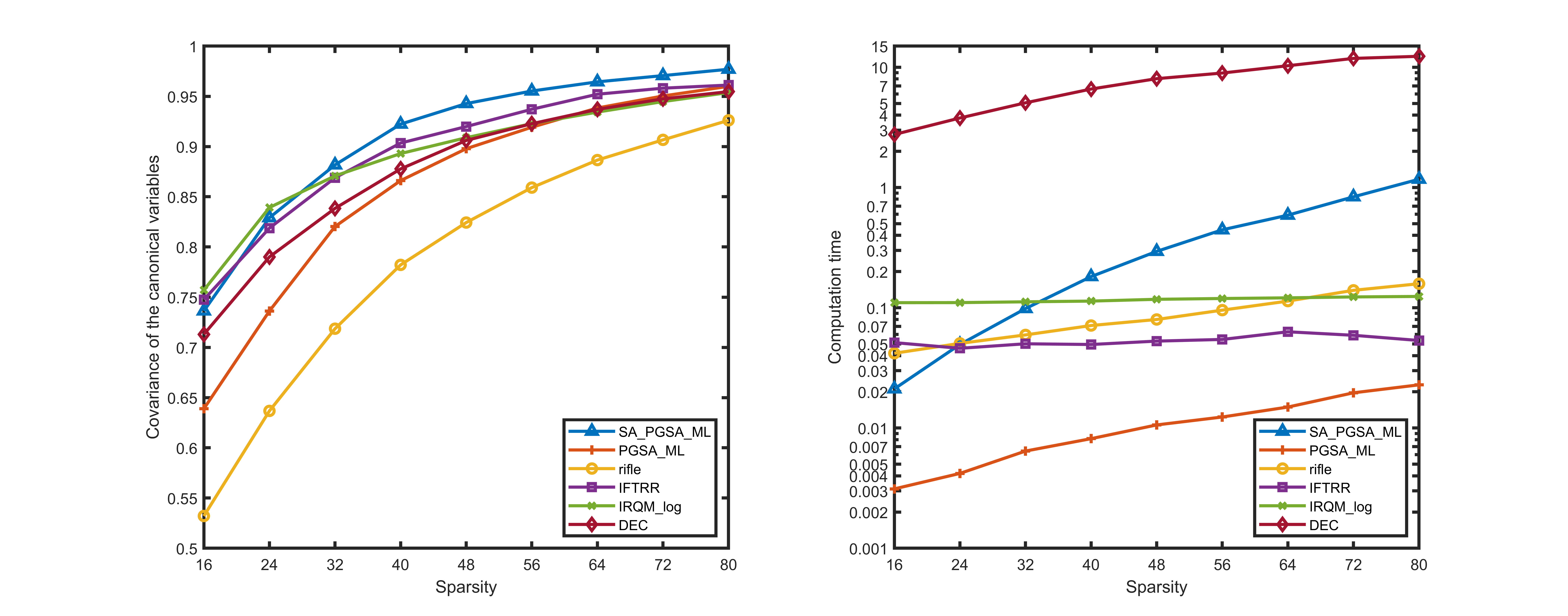}
	\caption{Covariance of the canonical variables vs sparsity (left), Computation time vs 
 sparsity (right)}
    \label{fig:sCCA}
\end{figure}

In terms of solution quality, SA\_PGSA\_ML obtains the highest covariance among all competing algorithms, with exceptions at sparsity levels $s=16$ and $s=24$. For these cases, although SA\_PGSA\_ML yields a lower covariance than IRQM\_log and IFTRR, it requires less computation time. When compared to PGSA\_ML — the gradient-based algorithm used in the first stage in SA\_PGSA\_ML — SA\_PGSA\_ML delivers a marked improvement across all sparsity levels, albeit at 7 to 50 times the computational cost. Furthermore, SA\_PGSA\_ML consistently outperforms DEC, by achieving superior covariance across all sparsity settings and requiring significantly less computation time.

\subsection{Sparse Sliced Inverse Regression}
We employ sparse sliced inverse regression on three real data sets including Prostate\_GE\cite{li2018feature}, RELATHE\cite{Web_qwone,Lang95}, and BASEHOCK\cite{Web_qwone,Lang95}. The comparisons of the objective value of the Rayleigh quotient for each algorithm on those data sets are shown in Figure~\ref{fig:Three-DataSets}.

\begin{figure}[htb]
 \centering
 \begin{tabular}{cc}
 \includegraphics[width=2.35in]{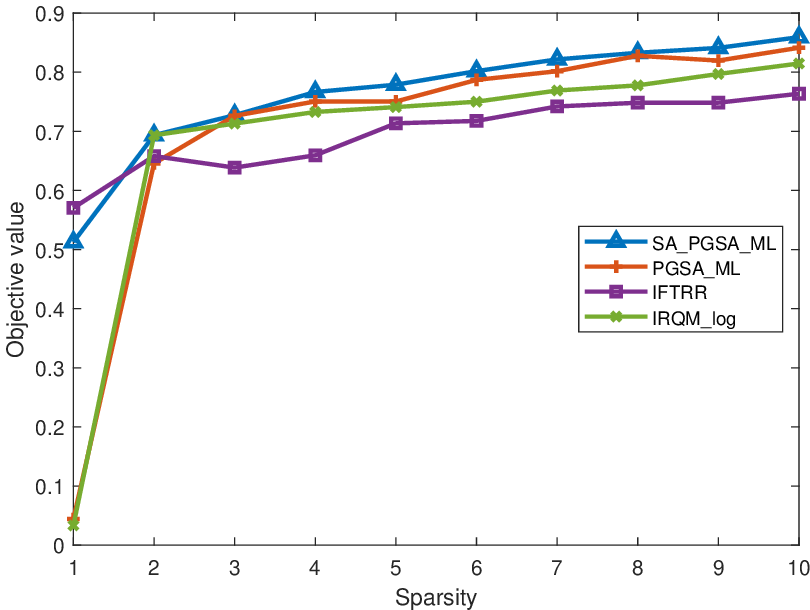}&
 \includegraphics[width=2.35in]{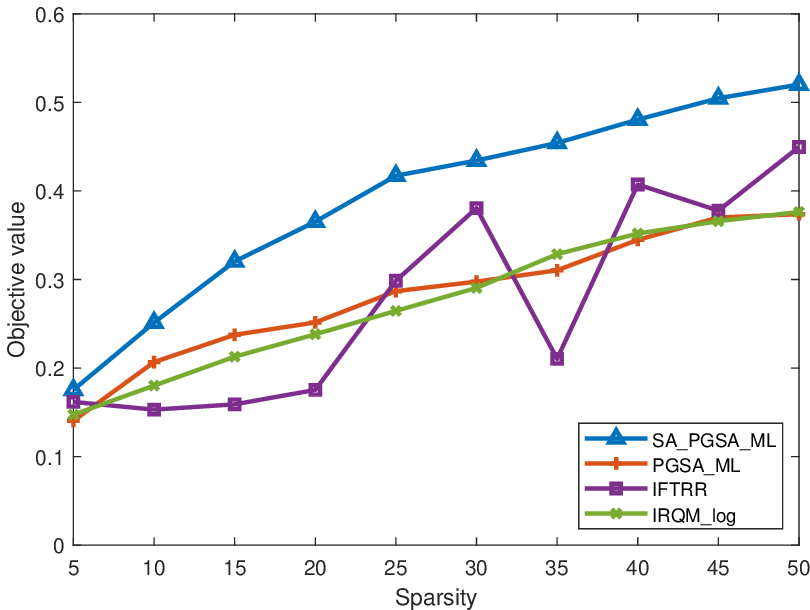}\\
 (a) the "Prostate\_GE" data set &(b) the "RELATHE" data set
 \end{tabular}
 \begin{tabular}{c}
 \includegraphics[width=2.35in]{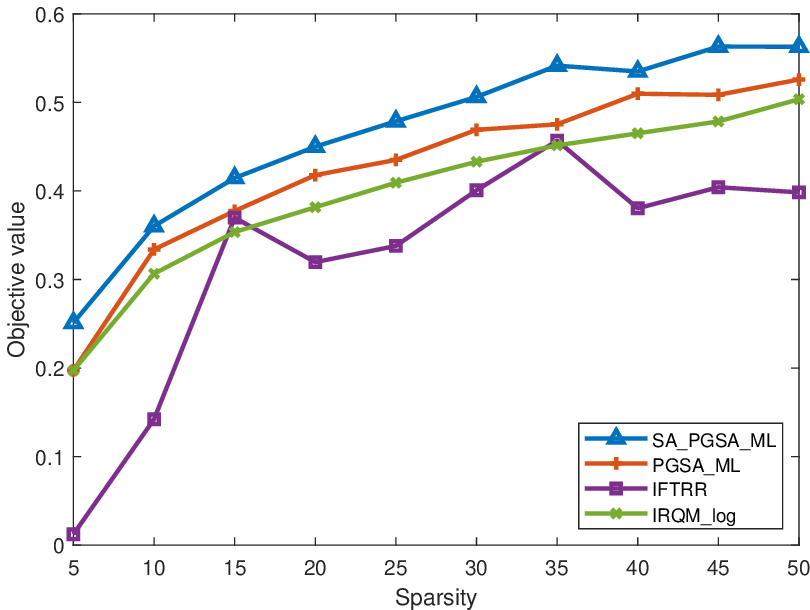}\\
 (c) the "BASEHOCK" data set
 \end{tabular}
 \caption{Objective value vs sparsity for (a) the "Prostate\_GE" data set; (b) the "RELATHE" data set; and (c) the "BASEHOCK" data set.}
 \label{fig:Three-DataSets}
 \end{figure}

Our algorithm consistently outperforms competing algorithms, with the exception of the $s=1$ case for the Prostate\_GE data set. It also significantly outperforms competing algorithms on RELATHE data set.

\section{Conclusions}\label{sec:conclusions}
In this work, we have addressed the computational challenges associated with the Sparse Generalized Eigenvalue Problem (sGEP), a cornerstone of several statistical learning methods. Recognizing the pitfalls of conventional approaches that are prone to local optima and sensitive initial conditions, we introduced the innovative \textit{support alteration} technique. This not only aids gradient-based iterations to overcome local optima but also provides a potential boost in the objective value. Harnessing this concept, we put forth a successive two-stage algorithm for sGEP. The algorithm alternates between using a gradient-based method to capture a stationary point and refining it through the \textit{support alteration} procedure. This iterative approach persists until specific stopping criteria are met. Our empirical evaluations underscore the superiority of our method over existing competitors, signaling its promise in the realms of statistical learning and optimization.

\section*{Acknowledgment of Support}
The work of Q. Li was supported in part by the Natural Science Foundation of China under grant 11971499 and the Guangdong Province Key Laboratory of Computational Science at the Sun Yat-sen University (2020B1212060032). 
The work of L. Shen was supported in part by the National Science Foundation under grant DMS-1913039 and DMS-2208385. 
The work of N. Zhang was supported in part by the National Science Foundation of China grant number 12271181, by the Guangzhou Basic Research Program grant number 202201010426 and by the Guangdong Basic and Applied Basic Research Foundation grant number 2023A1515030046.


\bibliography{allbib}
\bibliographystyle{siamplain}
\end{document}